\documentclass[12pt,thmsa]{amsart}
\usepackage{amssymb, euscript}

\def\Cal{\mathcal}

\def\B{{\Cal B}}

\def\M{{\Cal M}}

\def\D{{\Cal D}}
\def\S{{\Cal S}}

\def\K{{\Cal K}}
\def\I{{\Cal I}}

\def\Z{\mathcal{Z}}

\def\bbr{{\Bbb R}}

\def\bbn{{\Bbb N}}

\def\bbc{{\Bbb C}}

\def\supp{{\hbox{\rm supp}}}

\def\cl{{\hbox{\rm cl}}}

\def\cos{{\hbox{\rm cos}}}

\def\vol{{\hbox{\rm vol}}}
\def\min{{\hbox{\rm min}}}

\def\rl{\bbr^\ell}
\def\rn{\bbr^n}
\def\rnl{\bbr^{n-\ell}}

\def\part{\partial}
\def\intl{\int\limits}
\def\b{\beta}

\def\lng{\langle}
\def\rng{\rangle}
\def\Gam{\Gamma}
\def\Om{\Omega}
\def\a{\alpha}
\def\om{\omega}

\def\Del{\Delta}
\def\del{\delta}
\def\vp{\varphi}

\def\gam{\gamma}
\def\Lam{\Lambda}
\def\sig{\sigma}
\def\lam{\lambda}

\def\e{\varepsilon}

\newtheorem{theorem}{Theorem}[section]
\newtheorem{lemma}[theorem]{Lemma}
\newtheorem{definition}[theorem]{Definition}

\newtheorem{corollary}[theorem]{Corollary}
\newtheorem{proposition}[theorem]{Proposition}

\theoremstyle{remark}
\newtheorem{remark}[theorem]{Remark}

\newtheorem{example}[theorem]{Example}

\numberwithin{equation}{section}

\newcommand{\be}{\begin{equation}}
\newcommand{\ee}{\end{equation}}

\newcommand{\bea}{\begin{eqnarray}}
\newcommand{\eea}{\end{eqnarray}}
\newcommand{\Bea}{\begin{eqnarray*}}
\newcommand{\Eea}{\end{eqnarray*}}

\begin{document}

\title[Intersection Bodies ]
{Intersection Bodies and Generalized Cosine Transforms}

\author{Boris Rubin}
\address{
Department of Mathematics, Louisiana State University, Baton Rouge,
LA, 70803 USA}

%Institute of Mathematics, Hebrew University, Jerusalem 91904,
%ISRAEL}
\email{borisr@math.lsu.edu}

\thanks{The  research was supported in part by the  NSF grant DMS-0556157
and the Louisiana EPSCoR program, sponsored  by NSF and the Board of
Regents Support Fund.}

\subjclass[2000]{Primary 44A12; Secondary 52A38}

%\date{May, 10, 2005 and, in revised form 000.}

%\dedicatory{This paper is dedicated to our authors.}

\keywords{ Spherical Radon transforms, cosine transforms,
intersection bodies}

\begin{abstract}

Intersection bodies represent a remarkable class of geometric
objects associated with sections of star bodies and invoking
 Radon transforms,  generalized cosine
 transforms, and the relevant Fourier analysis.  The main focus of this article
 is interrelation between
    generalized cosine transforms of different kinds
 in the context of their application to investigation of a certain family of
 intersection bodies, which we call $\lam$-intersection bodies. The
 latter include  $k$-intersection bodies (in the sense of A. Koldobsky)
  and unit balls of finite-dimensional subspaces of
 $L_p$-spaces. In particular, we show that restrictions onto lower
 dimensional subspaces of the spherical Radon
 transforms and the generalized cosine transforms  preserve their integral-geometric structure. We apply this result to the
 study of sections of $\lam$-intersection bodies. New characterizations of
 this class of bodies are obtained and examples  are given.
 We also  review some known facts and give them new proofs.
\end{abstract}

\maketitle

\centerline{Contents} \centerline{} 1. Introduction.

2. Preliminaries.

3. Analytic families of the generalized cosine transforms.

4. Positive definite homogeneous distributions.

5. $\lam$-intersection bodies.

6. Examples of $\lam$-intersection bodies.

7. $(q,\ell)$-balls.

8. The generalized cosine transforms and comparison of volumes.

9. Appendix.

\section{Introduction}

\setcounter{equation}{0}

This is an updated and extended version of our previous preprint
\cite {R5}.

Intersection bodies interact with Radon transforms and encompass
diverse classes of geometric objects associated to sections of star
bodies. The concept of intersection body was introduced in the
remarkable paper by Lutwak \cite{Lu} and led to a breakthrough in
the solution of the long-standing Busemann-Petty problem;  see
\cite{G},  \cite{K3}, \cite{Lu}, \cite {Z2} for references and
historical notes.

We remind some known facts that will be needed in the following. An
origin-symmetric (o.s.) star body in $\bbr^n$, $ n \ge 2$, is  a
compact set $K$ with non-empty interior such that $tK \subset K \;
\forall t\in [0,1]$, $K=-K$, and the {\it radial function} $ \rho_K
(\theta) = \sup \{ \lambda \ge 0: \, \lambda \theta \in K \}$ is
continuous on the unit sphere $S^{n-1}$.
 In the following, $\K^n$ denotes
 the set of all o.s. star  bodies  in $\bbr^n$, $G_{n,i}$
 is the Grassmann manifold of $i$-dimensional linear
subspaces  of $\bbr^n$, and  $\vol_i ( \cdot)$ denotes the
$i$-dimensional volume function. The {\it Minkowski functional}  of
a body $K \in \K^n$  is defined by $ ||x||_K =\min \{a \ge 0 \, : \,
x \in aK\}$, so that $||\theta||_K=\rho_K^{-1}(\theta)$, $\theta \in
S^{n-1}$.
\begin{definition}\label{luu1} \cite{Lu} A body $K\in \K^n $ is an intersection
body of a body $L\in \K^n $ if $ \rho_K (\theta) =\vol_{n-1}(L\cap
\theta^\perp)$ for every $\theta \in S^{n-1}$, where $\theta^\perp$
is the central hyperplane orthogonal to $\theta$. \end{definition}

By taking into account that $\vol_{n-1}(L\cap \theta^\perp)$ in
Definition \ref{luu1} is a constant multiple of the Minkowski-Funk
transform $$(M f)(\theta)=\int_{S^{n-1}\cap \theta^\perp}
f(u)\,d_\theta u,\qquad f(u)= \rho_L^{n-1} (u),$$ Goodey, Lutwak and
Weil \cite {GLW} generalized Definition \ref{luu1}  as follows.
\begin{definition}\label{luu2}  A body $K\in \K^n $ is an
intersection body  if $ \rho_K  =M \mu$ for some
 even non-negative finite Borel measure $\mu$ on $S^{n-1}$.
\end{definition}

A sequence of bodies $K_j\in \K^n$ is said to be convergent to $K
\in \K^n$ in the radial metric if $\lim\limits_{j \to \infty}||
\rho_{K_j} - \rho_K||_{C(S^{n-1})}=0$.
\begin{proposition}\label{luu3}  The class of intersection bodies is the closure of the class of intersection bodies of
star bodies in the radial metric.
\end{proposition}
\begin{proposition} \label{luu4} If $K$ is an
intersection body in $\bbr^n$, $n>2$, then for every $i=2,3, \ldots
, n-1$ and every $\eta \in G_{n,i}$, $K\cap \eta$ is an intersection
body in $\eta$. \end{proposition} Regarding these two important
propositions see \cite {FGW}, \cite {GW} and a nice historical
survey in \cite {G}.

Different generalizations of the concept of intersection body
associated to lower dimensional sections were suggested in the
literature; see, e.g., \cite {K3}, \cite {RZ}, \cite {Z1}. The
following one, which  plays an important role in the study of the
lower dimensional  Busemann-Petty problem, is due to Zhang
\cite{Z1}.
\begin{definition}\label{zyy1} We say, that a body $K\in \K^n$
belongs to Zhang's class $\Z^n_{i}$ if there is a non-negative
finite Borel measure $m$ on the Grassmann manifold $G_{n,i}$ such
that $\rho_K^{n-i}=R^*_i m$, where $R^*_i$ is the dual spherical
Radon transform; see {\rm (\ref{rts}), (\ref{fina})}.
\end{definition}

Another generalization was suggested by Koldobsky \cite{K2} and
described in detail in  \cite{K3}. This class of bodies will be our
main concern.
\begin{definition}\label{kyy1}\cite[p. 71]{K3} A body $K\in
\K^n$ is a $k$-intersection body of a body $L\in \K^n$ $($we write
$K=\I\B_k (L))$
 if \be\label{divve} \vol_{k} (K\cap
\xi)=\vol_{n-k} (L\cap \xi^\perp)\qquad \forall \xi \in G_{n,k}.\ee
 We denote by   $\I\B_{k,n}$  the set of all bodies $K\in \K^n$ satisfying
(\ref{divve}) for some $L\in \K^n$.
\end{definition}
When $k=1$, this definition coincides with Definition \ref{luu1} up
to a constant multiple. An analog of Definition \ref{luu2} was given
 in the Fourier analytic terms as follows.
\begin{definition}\label{def45} \cite[Definition 4.7]{K3} A body $K\in
\K^n$  is a $k$-intersection body if  there is a non-negative finite
Borel measure $\mu$ on $S^{n-1}$, so that for every Schwartz
function $\phi$,
\[\int_{\bbr^n}||x||_K^{-k} \phi (x)\, dx=\int_{S^{n-1}}\Big
[\int_0^\infty t^{k-1}\hat \phi (t\theta) \,dt \Big ] \,d\mu
(\theta),\] where $\hat \phi$ denotes the Fourier transform of
$\phi$.

The set of all $k$-intersection bodies in $\bbr^n$ will be denoted
by $\I^n_k$.
\end{definition}

Keeping in mind  Proposition \ref{luu3} for $k=1$,  one
 can alternatively define the class $\I^n_k$  as a closure
of $\I\B_{k,n}$  in the radial metric; cf. \cite[p. 532]{Mi1}.
However,  to apply results from \cite{K3} to such class, equivalence
of this definition to Definition \ref{def45} must be proved. We will
do this in the more general situation in Section 5.2.

 From
Definitions \ref{kyy1} and \ref{def45} it is not clear, for which
bodies $L\in \K^n$ the relevant $k$-intersection body $K=\I\B_k (L)$
does exist. It is also not obvious which bodies actually  constitute
the class
 $\I^n_k$. The following important characterization is due to
 Koldobsky.
 \begin{theorem}\cite[Theorem 4.8]{K3}\label{tthm21} A body
$K\in \K^n$ is a $k$-intersection body if and only if
$||\cdot||_K^{-k}$ represents a positive definite tempered
distribution on $\bbr^n$, that is, the Fourier transform
$(||\cdot||_K^{-k})^\wedge$ is a positive tempered  distribution on
$\bbr^n $.
\end{theorem}

The concept of $k$-intersection body is  related to another
important development. For $K\in \K^n$, the quasi-normed space
$(\bbr^n, ||\cdot||_K)$ is said to be isometrically embedded in
$L_p, \; p>0$, if there is a linear operator $T: \bbr^n \to L_p
([0,1])$ so that $||x||_K=||Tx||_{L_p ([0,1])}$.
\begin{theorem}\cite[Theorem 6.10]{K3}\label{ttthm21}  The
space $(\bbr^n, ||\cdot||_K)$ embeds isometrically  in $L_p, \; p>0,
\; p\neq 2,4, \ldots\,$, if and only if $\Gam (-p/2)
(||\cdot||_K^p)^\wedge$ is a positive distribution on $\bbr^n
\setminus \{0\}$.
\end{theorem}
Following Theorems \ref{ttthm21} and \ref{tthm21},  one can formally
say that $K\in \I^n_k$  if and only if $(\bbr^n, ||\cdot||_K)$
embeds isometrically in $L_{-k}$. This observation, combined with
Definition \ref{def45},
 was used by A. Koldobsky to define the concept of ``isometric embedding in
 $L_p$'' for negative $p$.
\begin{definition}\label{def45k} \cite[ Definition 6.14]{K3} Let $0<p<n$, $K\in
\K^n$. The space $(\bbr^n, ||\cdot||_K)$ is said to be isometrically
embedded in $L_{-p}$ if  there is a non-negative finite Borel
measure $\mu$ on $S^{n-1}$, so that for every Schwartz function
$\phi$,
\[\int_{\bbr^n}||x||_K^{-p} \phi (x)\, dx=\int_{S^{n-1}}\Big
[\int_0^\infty t^{p-1}\hat \phi (t\theta) \,dt \Big ] \,d\mu
(\theta),\] where $\hat \phi$ denotes the Fourier transform of
$\phi$.
\end{definition}
Origin-symmetric bodies $K$  in this definition can be regarded as
``unit balls of $n$-dimensional subspaces of $L_{-p}$''. Comparing
Definitions \ref{def45k} and \ref{def45}, one might call these
bodies ``$p$-intersection bodies''.  Since the meaning of the space
$L_{-p}$  itself is not specified in Definition \ref{def45k} and
since our paper is mostly focused on geometric properties of bodies
(rather than
 embeddings in $L_{p}$), in the following we  prefer to
 adopt another name {\it ``$\lambda$-intersection body''}, where $\lam$
 is a real number, that will be specified in due course. We denote the set of all $\lam$-intersection
 bodies in $\bbr^n$ by $\I^n_\lam$.

{\bf Contents of the paper.} We will focus on  intimate connection
between intersection bodies, spherical Radon transforms, and
generalized cosine transforms; see definitions in Section 2.2. This
approach is motivated by the fact that the volume of a central cross
section of a star body is expressed through the spherical  Radon
transform, and the latter is a member of the analytic family of the
generalized cosine transforms. These transforms were introduced by
Semyanistyi \cite{Se} and arise (up to naming and normalization) in
different contexts of analysis and geometry; see, e.g., \cite{K3},
\cite{R3}-\cite{RZ}, \cite{Sa1}, \cite{Sa2}, \cite{Str1},
\cite{Str2}.

Sections 2-4 provide analytic background for geometric
considerations in Sections 5-7.   In Section 2 we establish our
notation  and define  the generalized cosine transforms on the
sphere and the relevant dual transforms on Grassmann manifolds. In
Section 3 we present basic properties of these transforms, establish
new relations between spherical Radon transforms and the generalized
cosine transforms, and prove ``restriction theorems'', which are
akin to trace theorems in Sobolev spaces. Section 4 deals with
positive definite homogeneous distributions, that can be
characterized in terms of the generalized cosine transforms.  This
section serves as a preparation for  the forthcoming definition of
the concept of $\lam$-intersection
 body. We investigate which $\lam$'s are appropriate and why. In
 Section 5 we switch to geometry and define the class $\I^n_\lam$
  of $\lam$-intersection  bodies. The case $0<\lam<n$ corresponds to
  the
 ``unit balls of $L_{-p}$-spaces'' in the spirit of Definition
 \ref{def45k}.  The reader will
find in this section new proofs of some known facts.  We introduce
the notion of {\it $\lam$-intersection body of a star body} in
$\bbr^n$, which extends Definition \ref{kyy1} to all $\lam <n, \;
\lam \neq 0$. The class of all such  bodies  will be denoted by
$\I\B^n_\lam$. We will prove that  for all $\lam<n$, $\lam \neq 0,
-2, -4, \ldots \,$, the class $\I^n_\lam$ is the closure of
$\I\B^n_\lam$ in the radial metric. The case $\lam =1$ gives
Proposition \ref{luu3}. It will be proved that all $m$-dimensional
central sections of $\lam$-intersection bodies are
$\lam$-intersection bodies in the corresponding $m$-planes provided
$\lam<m, \; \lam \neq 0$.

The natural question arises: How to construct $\lam$-intersection
bodies? In Section 6 we  give a series of examples; some of them are
known and some are new. They can be obtained by utilizing auxiliary
statements from Section 3. In particular, the famous embedding of
Zhang's class $\Z^n_{n-k}$ into $\I^n_k$, which was first
established in \cite{K4} and studied in \cite{Mi1}, \cite{Mi2}, will
be generalized to the case, when $k$ is replaced by any $\lam \in
(0,n)$. Section 7 is devoted to the so called $(q,\ell)$-balls,
defined by
$$ B^n_{q,\ell}=\{x=(x', x''): |x'|^q +|x''|^q \le 1; \;  x' \in
\rnl, \; x''\in \rl\}, \quad q>0.$$ We show that if $0<q \le 2$,
then $B^n_{q,\ell}\in \I^n_\lam$ for all $\lam \in (0,n)$. If $q>2$
and $n-3 \le \lam <n$, we still have $B^n_{q,\ell}\in \I^n_\lam$. If
$q>2$ and $0<\lam<\lam_0=\max (n-\ell, \ell)-2$, then
$B^n_{q,\ell}\not\in \I^n_\lam$. The case, when $q>2, \; \ell>1$,
and
 $\lam_0\le\lam <n-3$ represents an open problem.

 In Section 8 we remind the  generalized Busemann-Petty problem (GBP) for
 $i$-dimensional central sections of  o.s. convex
 bodies in $\bbr^n$. This challenging problem is still open for $i=2$ and $i=3$ ($n\ge
 5$). It actually inspires the whole investigation. Using properties of the generalized cosine transforms, we
 give a short direct proof of the  fact that an affirmative answer
 to GBP implies that every
smooth o.s.  convex body in $\bbr^n$ with positive curvature  is an
$(n-i)$-intersection body. This fact was discovered by A. Koldobsky.
 The original proof in \cite{K4} is based on the embedding $\I^n_{n-i} \subset \Z^n_i$
 and Zhang's result \cite [Theorem 6]{Z1}. The latter heavily relies on the Hahn-Banach separation theorem.
Our proof is  more constructive and almost self-contained. We
conclude the paper by Appendix, which is added for convenience of
the reader.

The list of references at the end of the paper is far from being
 complete. Further references can be found in cited books and papers.

 {\bf Acknowledgement.} I am grateful to  Professor
Alexander Koldobsky, who shared with me his knowledge of the
subject. Special thanks go to Professors Erwin Lutwak, Deane Yang,
and Gaoyong Zhang for useful discussions.

\section{Preliminaries}

\subsection{Notation.} In the following, $\bbn=\{1,2, \ldots \,\}$ is
the set of all natural numbers,
 $S^{n-1}$ is the unit sphere in $\bbr^n$ with  the area $\, \sig_{n-1}=
2\pi^{n/2}/\Gam (n/2) $; $C_e(S^{n-1})$ is the space of even
continuous functions on
 $S^{n-1}$; $SO(n)$ is the special orthogonal group of
$\bbr^n$; for $\theta \in S^{n-1}$ and $\gam \in SO(n)$,  $d\theta$
and $d\gam$ denote the relevant invariant probability measures;
$\D(S^{n-1})$ is the space of $C^\infty$-functions on $S^{n-1}$
equipped with the standard topology, and $\D'(S^{n-1})$ stands for
the corresponding dual space of distributions. The subspaces of even
test functions (distributions) are denoted by $\D_{e}(S^{n-1})$ (
$\D_{e}'(S^{n-1})$); $G_{n,i}$ denotes the Grassmann manifold of
$i$-dimensional subspaces $\xi$ of $\Bbb R^n$ with the
$SO(n)$-invariant probability measure $d\xi$; $\D(G_{n,i})$ is the
space of infinitely differentiable functions on $G_{n,i}$.

We write  $\M(S^{n-1})$ and $\M(G_{n,i})$ for the spaces  of  finite
Borel measures on $S^{n-1}$ and $G_{n,i}$;
 $\M_+(S^{n-1})$  and $\M_+(G_{n,i})$ are the relevant spaces of
non-negative  measures;  $ \M_{e+}(S^{n-1})$ denotes the space of
even
 measures $\mu \in \M_+(S^{n-1})$. Given a function $\vp$ on $G_{n,i}$, we denote $\vp^\perp
(\eta)=\vp(\eta^\perp), \quad \eta \in G_{n,n-i}$. Similarly, given
a measure $\mu \in \M(G_{n,n-i})$, the corresponding ``orthogonal
measure''  $\mu^\perp$ in $ \M(G_{n,i})$ is defined by $(\mu^\perp,
\vp)=(\mu,\vp^\perp)$, $\vp \in C(G_{n,i})$.

Let $\{ Y_{j, k}  \}$ be an orthonormal basis of spherical harmonics
on $S^{n-1}$. Here $j = 0, 1, 2, \dots ,$ and $k = 1, 2, \dots, d_n
(j)$, where $d_n (j)$  is the dimension of the subspace of spherical
harmonics of degree $j$. Each function $\om \in \D(S^{n-1})$ admits
a decomposition $\om =\sum_{j,k} \om_{j,k}Y_{j, k}$ with the
Fourier-Laplace coefficients $\om_{j,k}=\int_{S^{n-1}}\om
(\theta)Y_{j, k}(\theta)d\theta$, which decay rapidly as $j\to
\infty$. Each distribution $f \in \D'(S^{n-1})$ can be defined by
$(f,\om)=\sum_{j,k} f_{j,k}\om_{j,k}$ where
 $f_{j,k}=(f,Y_{j, k})$ grow not faster than $j^m$ for some
integer $m$. We will need the Poisson integral, which is defined for
$f\in L^1 (S^{n-1})$  by \be\label{pu} (\Pi_t f)(\theta)= (1-t^2)
\int_{S^{n-1}} f(u) |\theta-t u|^{-n} du, \quad 0<t <1,\ee and has
the Fourier-Laplace decomposition $\Pi_t f=\sum_{j,k} t^j
f_{j,k}Y_{j, k}$ \cite {SW}. For $f \in \D'(S^{n-1})$, this
decomposition serves as a definition of $\Pi_t f$. For harmonic
analysis on the unit sphere, the reader is referred to \cite{Le},
\cite{M}, \cite{Ne}, \cite{SW}, and a survey article \cite{Sa2}.

\subsection{Basic integral transforms} For integrable functions $f$ on
$S^{n-1}$ and $\varphi $ on $G_{n,i}$, $1\le i\le n-1$,  the
spherical Radon transform $(R_i f)(\xi), \, \xi \in G_{n,i}$,
 and its dual  $(R_i^*\varphi)(\theta), \, \theta \in S^{n-1}$,  are defined by \be\label{rts}
 (R_i f)(\xi) = \int_{\theta \in S^{n-1}\cap\xi} f(\theta) \, d_\xi \theta, \qquad
  (R_i^* \varphi)(\theta) = \int_{\xi \ni \theta}  \varphi (\xi)  \, d_\theta \xi,
\ee where $d_\xi \theta$ and $ d_\theta \xi$ denote the probability
measures on the  manifolds $S^{n-1}\cap\xi$ and $\{\xi\in G_{n,i}:
\xi \ni \theta \}$, respectively. The precise meaning of the second
integral is
 \be\label{drt}
 (R_i^*
 \vp)(\theta)=\int_{SO(n-1)}\vp (r_\theta \gam p_0) \,  d\gam, \qquad \theta \in
 S^{n-1},
 \ee
 where $p_0$ is an arbitrarily fixed coordinate
$i$-plane containing the north pole $ e_{n}$  and $ r_\theta \in
SO(n)$ is a rotation satisfying $r_\theta e_n =\theta$.

Operators $R_i$ and $R_i^*$ extend to finite Borel measures in a
canonical way, using the duality \be\label{dual} \int_{G_{n,i}}
(R_if)(\xi) \vp (\xi) d \xi =
 \int_{S^{n-1}} f(\theta) (R_i^* \vp) (\theta)
d\theta. \ee Specifically, for $\mu \in \M(S^{n-1})$ and $m \in
\M(G_{n,i})$,  we define $R_i\mu \in \M(G_{n,i})$ and $R_i^* m \in
\M(S^{n-1})$ by \be\label{fina} (R_i\mu,
\vp)\!=\!\!\int_{S^{n-1}}\!(R_i^* \vp)(\theta)d\mu (\theta), \quad
(R_i^* m, f)\!=\!\!\int_{G_{n,i}}\!(R_if)(\xi) dm(\xi),\ee where
$\vp \in C(G_{n,i}), \; f \in C(S^{n-1})$.

 The generalized cosine transforms are defined by
 \be\label{rka} (R_i^\a f)(\xi)=\gam_{n,i}(\a)\,
\int_{S^{n-1}} |\text{\rm Pr}_{\xi^\perp} \theta|^{\a+i-n} \,
f(\theta) \, d\theta,\ee
 \be\label{rkad}
 (\overset *  R_i {}^\a \vp)(\theta)=\gam_{n,i}(\a)\,\int_{G_{n,i}}  |\text{\rm
Pr}_{\xi^\perp} \theta|^{\a+i-n} \, \vp (\xi)\, d\xi, \ee
$$\gam_{n,i}(\a)=\frac{ \sig_{n-1}\,\Gamma((n - \a- i)/2)} {2\pi^{(n-1)/2} \,
\Gamma(\a/2)}, \qquad Re \, \a
>0, \quad \a+i-n \neq 0,2,4, \ldots .$$
Here  $\text{\rm Pr}_{\xi^\perp} \theta $ stands for
 the  orthogonal
 projection of $\theta$ onto $\xi^\perp$, the orthogonal
 complement of $\xi\in G_{n,i}$.
 If $f$ and $\vp$ are smooth enough, then integrals
(\ref{rts}) can be regarded (up to a constant multiple) as members
of the relevant analytic families  (\ref{rka}) and (\ref{rkad});
 cf. Lemma \ref{lar}.
 The particular case $i=n-1$ in (\ref{rts}) corresponds to the
Minkowski-Funk transform \be\label{mf} (M f)(u)=\int_{\{\theta \,:
\, \theta \cdot u =0\}} f(\theta) \,d_u\theta=(R_{n-1} f)(u^\perp),
\qquad u \in S^{n-1}, \ee which integrates a function $f$  over
great circles of codimension $1$. This transform is a member of the
analytic family \be\label{af} (M^\a f)(u)=(R_{n-1}^\a f)(u^\perp)=
\gam_n(\a)\, \int_{S^{n-1}} f(\theta) |\theta \cdot u|^{\a-1}
\,d\theta, \ee \be\label{beren} \gam_n(\a)\!=\!{
\sig_{n-1}\,\Gamma\big( (1\!-\!\a)/2\big)\over 2\pi^{(n-1)/2} \Gamma
(\a/2)}, \qquad Re \, \a \!>\!0, \quad \a \!\neq \!1,3,5, \ldots
.\ee The values $\a=1,3,5, \ldots$ are  poles of the Gamma function
$\Gamma ( (1\!-\!\a)/2)$. In some occasions we include these values
into consideration and  set \be\label{afu} (\tilde M^\a f)(u)=
\int_{S^{n-1}} f(\theta) |\theta \cdot u|^{\a-1} \,d\theta. \ee

{\bf Historical notes.} Regarding spherical Radon transforms
(\ref{rts}) and the Minkowski-Funk transform (\ref{mf}), see
\cite{GGG}, \cite{He}, \cite{R1}, \cite{R2}. The first detailed
investigation of the analytic family
 $\{M^\a \}$ is due to  Semyanistyi \cite{Se}, who showed that
 these operators naturally arise in the Fourier analysis of
 homogeneous functions. The case $\a=2$ in (\ref{afu}) was known before, thanks to W. Blaschke, A.D. Alexandrov, and
P. L\'evy. Integrals (\ref{af})  (sometimes with different
normalization) arise in diverse areas of analysis and geometry; see
\cite{K3}, \cite{R3} - \cite{R2}, \cite{Sa2}, \cite{Str1}, and
references therein. In
 convex geometry and  Banach space theory, operators
(\ref{afu}) with $\a-1$ replaced by $p$  are known as the $p$-cosine
transforms. More general analytic families (\ref{rka}) and
(\ref{rkad}) were introduced in \cite {R1}.

\section{Analytic Families of the Generalized Cosine Transforms}

\setcounter{equation}{0}

\subsection{Basic properties} Below we review basic properties of integrals
(\ref{rka}), (\ref{rkad}), (\ref{af}); see \cite {R1}, \cite {R2}
for more details. For  integrable functions $f$ and $\vp$ and $Re \,
\a>0$, integrals (\ref{rka}), (\ref{rkad}) and (\ref{af}) are
absolutely convergent. When $f$ and $\vp$ are infinitely
differentiable, these integrals extend meromorphically to all $\a
\in \bbc$.
\begin{lemma}\label{lar} If $f$ and $\vp$ are continuous functions, then
\bea\label{lim} \lim\limits_{\a \to 0} R_i^\a f&=&R_i^0 f =c_i
\,R_if,
\qquad c_i=\frac{\sig_{i-1}}{2\pi^{(i-1)/2}};\\
\label{limd} \lim\limits_{\a \to 0} \overset *  R_i {}^\a \vp
&=&\overset *  R_i {}^0 \vp =c_i\, R_i^* \vp,\\
\label{lim1} \lim\limits_{\a \to 0} M^\a f&=&M^0 f= c_{n-1}\,
 Mf, \qquad c_{n-1}=\frac{\sig_{n-2}}{2\pi^{(n-2)/2}}.\eea
 Hence, the Radon transform, its dual, and the Minkowski-Funk
transform can be regarded (up to a constant multiple) as  members of
the corresponding analytic families $\{ R_i^\a\}$, $\{\overset * R_i
{}^\a\}$, $\{ M^\a\}$. \end{lemma}
\begin{proof}
Formulas (\ref{limd}) and (\ref{lim1})  follow from (\ref{lim}). To
prove (\ref{lim}), we write  (\ref{rka}) in  bi-spherical
coordinates
 $\theta =u\sin\, \psi +v\cos\,
\psi$, where
$$u\in S^{n-1} \cap  \xi \sim S^{i-1}, \quad v \in
S^{n-1} \cap \xi^\perp \sim S^{n-i-1}, \quad 0\le \psi\le\pi/2.$$
$$
d\theta=c\, \sin^{i-1} \psi \,\cos^{n-i-1} \psi \,d\psi du dv, \quad
c=\sig_{i-1}\sig_{n-i-1}/\sig_{n-1}.$$ This gives \bea (R_i^\a
f)(\xi)&=&c\,\gam_{n,i}(\a) \int_0^{\pi/2 }\sin^{i-1}\psi \,
\cos^{\a-1}\psi \,d\psi\nonumber\\&\times&\int_{S^{n-1} \cap
\xi^\perp} dv\int_{S^{n-1} \cap  \xi} f(u\sin\, \psi
\!+\!v\cos\, \psi)\, du\nonumber\\
&=&\frac{c_i (\a)}{\Gamma(\a/2)}\,\int_0^1 t^{\a/2
-1}F(t)\,dt,\nonumber\eea where
$$
c_i (\a)=\frac{c\,\gam_{n,i}(\a)\, \Gamma(\a/2)}{2}=
\frac{\sig_{i-1}\sig_{n-i-1}}{2} \, \frac{\Gamma((n - \a- i)/2)}
{2\pi^{(n-1)/2}} \, \to \, \frac{\sig_{i-1}}{2\pi^{(i-1)/2}}$$ as
$\a \to 0$, and
$$ F(t)=(1-t^2)^{i/2 -1}\int_{S^{n-1} \cap
\xi^\perp} dv\int_{S^{n-1} \cap  \xi}f(u\sqrt{1-t^2}\!+\!vt)\, du.
$$ Since
$$\lim\limits_{\a \to 0}\frac{1}{\Gamma(\a/2)}\int_0^1
t^{\a/2 -1}F(t)\,dt=F(0)=\int_{S^{n-1} \cap  \xi}f(u)d u=(R_i
f)(\xi),$$ we are done. \end{proof}

Analytic continuation of integrals (\ref{af}) can be realized in
spherical harmonics as $ M^\a f\!=\!\sum\limits_{j, k} m_{j,\a}
f_{j, k} Y_{j, k}$, where \be\label{55} m_{j, \a} \!=\!\left\{
\begin{array}{cl} \!(-1)^{j/2}\, \displaystyle{\frac{\Gamma
(j/2+(1-\a)/2)}{ \Gamma (j/2+(n-1+\a)/2)}}
  &  \mbox{\rm if $j$ is even}, \\
0 &  \mbox{\rm if $j$ is odd};
\end{array}
\right.  \ee see \cite{R3}, \cite{R2}. If $f\!\in \! \D'(S^{n-1})$,
then $M^\a f$ is a distribution defined by $$(M^\a
f,\om)\!=\!(f,M^\a\om)\!=\!\sum\limits_{j,k} m_{j,\a} \,f_{j,
k}\,\om_{j, k},\quad \om \!\in \!\D(S^{n-1}); \quad\a \!\neq
\!1,3,5, \ldots \,.$$
\begin{lemma}\label{l1} Let $\a, \b \in \bbc; \; \a, \b \neq
1,3,5, \ldots \,$. If $\a+\b=2-n$ and $f\in \D_e(S^{n-1})$ $($or $f
\in \D'_e(S^{n-1}))$, then \be\label{st}M^\a M^{\b}f=f.\ee If $\,
\a, 2-n-\a \neq 1,3,5, \ldots $, then $M^\a$ is an automorphism of
the spaces $\D_e(S^{n-1})$ and $\D'_e(S^{n-1})$.
\end{lemma}
\begin{proof} The equality (\ref{st}) is
equivalent to $m_{j, \a} m_{j, \b}=1$, $\;\a+\b=2-n$. The latter
 follows
 from (\ref{55}). The second statement  is a
consequence of the standard theory of spherical harmonics \cite{Ne},
because the Fourier-Laplace multiplier $m_{j,\a}$  has a power
behavior as $j \to \infty$.
\end{proof}
\begin{corollary} The Minkowski-Funk transform on the spaces
$\D_e(S^{n-1})$  and $\D'_e(S^{n-1})$
 can be inverted by the formula
\be\label{mmm} (M)^{-1}=c_{n-1}\,M^{2-n}, \qquad
c_{n-1}=\frac{\sig_{n-2}}{2\pi^{(n-2)/2}}.\ee
\end{corollary}

Note that there is a wide variety of diverse inversion formulas for
the Minkowski-Funk transform (see \cite{GGG}, \cite{He}, \cite{R2}
and references therein), but all of them
 are, in fact, different realizations of (\ref{mmm}), depending on
 classes of functions.

\subsection{Auxiliary statements}
We establish some connections between operator families defined
above.
\begin{lemma}\label{lcon} Let $\a, \b \in \bbc; \; \a, \b \neq
1,3,5, \ldots \,$. If $Re\, \a >Re\, \b$, then $M^\a=M^\b A_{\a,\b
}$, where $A_{\a,\b }$ is a spherical convolution operator  with the
Fourier-Laplace multiplier \be\label{ab} a_{\a,\b }(j)= \frac{\Gamma
(j/2+(1-\a)/2)}{ \Gamma (j/2+(n-1+\a)/2)}\, \frac{\Gamma
(j/2+(n-1+\b)/2)}{ \Gamma (j/2+(1-\b)/2)},\ee so that $a_{\a,\b
}(j)\sim (j/2)^{\b-\a}$ as $ j \to \infty$. If $\a$ and $ \b$ are
real numbers satisfying $\a> \b>1-n,\; \a+\b<2$, then $A_{\a,\b }$
is
 an integral operator such that $A_{\a,\b }f\ge 0$ for every
 non-negative  $f \in L^1 (S^{n-1})$.
\end{lemma}
\begin{proof} The first statement follows from (\ref{55}). To prove
the second one, we consider  integral operators \bea\label{q1}\quad
(Q_{+}^{\mu,\nu} f)(x) &=& \frac{2}{\Gam (\mu/2)} \int_0^1
(1 - t^2)^{\mu/2 -1} (\Pi_t f)(x) \, t^{n-\nu} dt,\\
\label{q2}\quad (Q_{-}^{\mu,\nu} f)(x) &=& \frac{2}{\Gam (\mu/2)}
\int_1^\infty (t^2-1)^{\mu/2 -1} (\Pi_{1/t} f)(x) \, t^{1-\nu}
dt,\eea expressed through the Poisson integral (\ref{pu}). The
Fourier-Laplace multipliers of $Q_{+}^{\mu,\lambda}$ and
$Q_{-}^{\mu,\nu}$ are
 \be\label{mqa} \hat q_{+}^{\mu,\nu}(j)\!=
\!\frac{\Gam((j\!+\!n\!-\!\nu \!+\!1)/2)}{\Gam((j\!+\!n\!-\!\nu
\!+\!1\!+\!\mu)/2)}, \quad
 \hat  q_{-}^{\mu,\nu}(j)\!=\!
\frac{\Gam((j\!+\!\nu\! -\!\mu)/2)}{\Gam((j\!+\!\nu )/2)}.\ee They
can be easily computed by taking into account that $\Pi_t \sim t^j$
in the Fourier-Laplace terms. If $f\in L^1(S^{n-1})$ and  $0< \mu<
\nu<n$, then integrals (\ref{q1}) and (\ref{q2}) are absolutely
convergent  and obey $Q_{\pm}^{\mu,\nu} f\ge 0$ when   $f\ge 0$.
 Comparing (\ref{mqa}) and (\ref{ab}), we
obtain a factorization $A_{\a,\b }=Q_{+}^{\a-\b,1-\b}
Q_{-}^{\a-\b,1-\b}$ (set $\mu=\a-\b,\; \nu=1-\b$), which implies the
second statement of the lemma.
\end{proof}

It is convenient to introduce a special notation for the  spherical
Radon transform and the generalized cosine transform with orthogonal
argument. Assuming $\xi \in G_{n,i}$, we denote \be\label{prp}
(R_{n-i,\perp}f)(\xi)=(R_{n-i}f)(\xi^{\perp}), \qquad
(R^\a_{n-i,\perp}f)(\xi)=(R^\a_{n-i}f)(\xi^{\perp}).\ee
\begin{lemma}\label{l2} Let $f \in L^1(S^{n-1}), \quad Re\, \a >0; \quad \a \neq 1,3,5, \ldots \,
$. Then \be \label{con} (R_i M^\a f)(\xi)=c\,
 (R_{n-i,\perp}^{\a +i-1}
f)(\xi), \qquad \xi \in G_{n,i},\quad c=
\frac{2\pi^{(i-1)/2}}{\sig_{i-1}},\ee or (replace $i$ by
 $n-i$)\be \label{conn} (R_{n-i,\perp} M^\a f)(\xi) = \frac{2\pi^{(n-i-1)/2}}{\sig_{n-i-1}} \,
(R_i^{\a +n-i-1} f)(\xi).\ee If
 $f \in \D_e(S^{n-1})$, then (\ref{con}) and (\ref{conn}) extend to $Re\, \a \le 0$ by
 analytic continuation.
\end{lemma}
\begin{proof} For $Re\, \a >0$,
\[(R_i M^\a f)(\xi)=\gam_n(\a) \int_{S^{n-1}\cap\xi} \, d_{\xi}
u  \int_{S^{n-1}} f(\theta) |\theta\cdot u|^{\a-1} \,d\theta.\]
Since $|\theta\cdot u|=|\text{\rm Pr}_{\xi} \theta||v_\theta \cdot
u|$ for some $v_\theta \in S^{n-1}\cap\xi$, by changing the order of
integration, we obtain \[ (R_i M^\a f)(\xi)=\gam_n(\a)\,
\int_{S^{n-1}} f(\theta) |\text{\rm Pr}_{\xi} \theta|^{\a-1}
\,d\theta  \int_{S^{n-1}\cap\xi} |v_\theta \cdot u|^{\a-1} d_{\xi}
u.\] The inner integral is independent on $v_\theta$ and can be
easily evaluated: \bea \int_{S^{n-1}\cap\xi} |v_\theta \cdot
u|^{\a-1} d_{\xi} u&=&\frac{\sig_{i-2}}{\sig_{i-1}}\int_{-1}^1
|t|^{\a-1} (1-t^2)^{(i-3)/2}\, dt \nonumber\\
&=&\frac{2\pi^{(i-1)/2} \,\Gamma (\a/2)}{ \sig_{i-1}\,\Gamma(
(i+\a-1)/2)}.\nonumber\eea
 This implies (\ref{con}).
 \end{proof}

The following  statement is dual to Lemma \ref{l2}.
\begin{lemma}\label{l4} Let $\mu \in \M(G_{n,i}),\; \a \neq 1,3,5, \ldots \, $. Then \be
\label{cndmi} M^\a R_i^* \mu= c\, \overset * R{}^{\a +i-1}_{n-i}
\mu^\perp,\qquad c= 2\pi^{(i-1)/2}/\sig_{i-1},\ee in the
$\D'(S^{n-1})$-sense.
 If $Re\, \a
>0$ and $\mu$ is absolutely continuous with density $\vp \in
 L^1(G_{n,i})$, then
\be \label{cnd} M^\a R_i^* \vp= c\, \overset * R{}^{\a +i-1}_{n-i}
\vp^\perp\ee almost everywhere on $S^{n-1}$.
 If $\vp
\in \D(G_{n,i})$, then (\ref{cnd}) extends to all complex $\a\neq
1,3,5, \ldots $ by analytic continuation.
\end{lemma}
\begin{proof}
Let  $\om \in \D_e (S^{n-1})$ (it suffices to consider only even
test functions). By (\ref{dual}) and (\ref{con}),
$$
(M^\a R_i^* \mu, \om)=(\mu,  R_i M^\a \om)=c\,(\mu,
 R_{n-i,\perp}^{\a +i-1} \om)=c\, (\mu^\perp, R_{n-i}^{\a +i-1} \om).$$
 This gives the  result.
\end{proof}

 The next statement  contains explicit representations of the right
inverse of the dual Radon transform  $R_i^*$ (note that $R_i^*$ is
non-injective on $\D (G_{n,i})$ when $1<i<n-1$).

\begin{lemma}\label{l34} Every  function $f\!\in \!\D_e (S^{n-1})$ is
represented as $f\!=\!R_i^* Af$, where $A:\D_e (S^{n-1}) \to
\D(G_{n,i})$, \be\label{crl} Af=c_1\,R_i^{1-i}f=c_2\,R_{n-i,\perp}
M^{2-n}f,\ee
$$
c_1=\frac{\pi^{(1-i)/2}\sig_{n-2}}{\sig_{n-i-1}}=\frac{\Gam((n-i)/2)}{\Gam((n-1)/2)},
\qquad c_2=\frac{\sig_{n-2}}{2\pi^{n/2-1}}.$$
\end{lemma}
\begin{proof}  The coincidence of expressions in (\ref{crl})  follows
from (\ref{conn}). To prove the first equality, we invoke spherical
convolutions defined by analytic continuation of the integral
\be\label{sin} (Q^\a f)(\theta)\!=\!\frac{\sig_{n-1}\Gam
((n\!-\!1\!-\!\a)/2) } {2 \pi^{(n-1)/2} \Gam
(\a/2)}\!\int_{S^{n-1}}\!\!(1\!-\!|u \cdot \theta|^2)^{(\a-n+1)/2}
f(u) du,\ee $Re \,\a > 0, \quad \a-n \neq 0,2,4, \ldots \, $, so
that  $Q^0 f =f$ \cite{R1}. By Theorem 1.1 from \cite{R1}, $ R^*_i
R_i^\a f=c_1^{-1} Q^{\a+i-1}f$, and therefore (set $\a=1-i$), $R^*_i
R_i^{1-i} f=c_1^{-1}  f$, as desired.
\end{proof}

The next statement provides an intriguing factorization of the
 Minkowski-Funk transform in terms of Radon transforms associated to mutually orthogonal
 subspaces. This factorization can be useful in different
 occurrences.
\begin{theorem}\label{l3} For $f \in L^1(S^{n-1})$ and $0<i<n$,
\be\label{svv} Mf=R_i^* R_{n-i,\perp} f.\ee
\end{theorem}
\begin{proof} By (\ref{drt}),
\bea (R_i^* R_{n-i,\perp} f)(\theta)&=&\intl_{SO(n-1)}
(R_{n-i,\perp} f)(r_\theta \gam \bbr^i)\, d\gam\nonumber \\
&=&\intl_{SO(n-1)} (R_{n-i}
f)(r_\theta \gam \bbr^{n-i})\, d\gam \nonumber \\
&=& \intl_{SO(n-1)}d\gam \intl_{S^{n-1}\cap r_\theta \gam
\bbr^{n-i}} f(v) \, dv \nonumber \\
&=& \intl_{S^{n-1}\cap \bbr^{n-i}} dw \intl_{SO(n-1)}f(r_\theta \gam
w) \,d\gam.\nonumber \eea The inner integral is independent on $w
\in S^{n-1}\cap \bbr^{n-i}$ and equals $(M f)(\theta)$. This gives
(\ref{svv}).
\end{proof}

\subsection{Restriction theorems}
Theorems of such type deal with traces of functions on lower
dimensional subspaces and are well known, for instance, in the
theory of function spaces. To the best of our knowledge,
 traces of functions represented by
 Radon transforms or, more generally, by the generalized cosine
 transforms , were not studied systematically and deserve particular
 attention, because they
provide analytic background
 to a series of results related to sections of star bodies; cf.
 \cite[Sec. 3.5]{R2}, \cite{FGW}. Given a subspace $\eta \in
 G_{n,m}$ and $k<m$, we denote by $G_k(\eta)$ the manifold of all
 $k$-dimensional subspaces of $\eta$.

\begin{theorem}\label{restrf} Let $f\in C_e(S^{n-1})$, $1\le
k<m<n$, $\lam \neq 0, -2, -4, \ldots \,$. If $Re \, \lam<k$, then
for every $\eta \in G_{n,m}$ and every $\xi \in G_k(\eta)$,
\be\label{yab1} (R_{n-k}^{k-\lam}
f)(\xi^\perp)=(R_{m-k}^{k-\lam}T_\eta ^\lam f)(\xi^\perp \cap
\eta),\ee  where \be\label{teta} (T_\eta^\lam f)(u)=\tilde
c\!\!\!\intl_{S^{n-1} \cap (\eta^\perp \oplus \bbr u )} \!\!\!f (w)
|u\cdot w|^{m-\lam-1}\, dw, \ee
$$  u \in S^{n-1} \cap
\eta, \qquad \tilde c=\pi^{(m-n)/2}\, \sig_{n-m}/2.$$ In particular
(let $\lam \to k$), \be\label{yab2}
(R_{n-k}f)(\xi^\perp)\!=\!c\,(R_{m-k}T_\eta^k f)(\xi^\perp \cap
\eta),\quad
c\!=\!\frac{\pi^{(n-m)/2}\,\sig_{m-k-1}}{\sig_{n-k-1}}.\ee
\end{theorem}
\begin{proof}  By (\ref{rka}),
$$(R_{n-k}^{k-\lam}f)(\xi^\perp)\!=\!\gam_{n,n-k}(k-\lam)\,
\int_{S^{n-1}} |\text{\rm Pr}_{\xi} \theta|^{-\lam} \, f(\theta) \,
d\theta. $$ We represent $\theta$ in bi-spherical coordinates as
 \be\label{bsc}\theta =u\cos\, \psi +v\sin \psi,\ee where $$u\in S^{n-1} \cap  \eta \sim S^{m-1}, \quad v \in
S^{n-1} \cap \eta^\perp \sim S^{n-m-1}, \quad 0\le \psi\le\pi/2,$$
$$
d\theta=c''\, \sin^{n-m-1}\psi \, \cos^{m-1}\psi \,d\psi du dv,
\quad c''=\sig_{m-1}\sig_{n-m-1}/\sig_{n-1}.$$ If $\xi \subset
\eta$, then $|\text{\rm Pr}_{\xi} \theta|=|\text{\rm Pr}_{\xi}
[\text{\rm Pr}_{\eta}\theta ]|=|\text{\rm Pr}_{\xi} u| \, \cos
\,\psi$, and therefore,
$$
(R_{n-k}^{k-\lam}f)(\xi^\perp)=\gam_{m,m-k}(k-\lam)\, \int_{S^{n-1}
\cap  \eta}|\text{\rm Pr}_{\xi} u|^{-\lam} (T_\eta^\lam f)(u)\,
du,$$ where\bea (T_\eta^\lam f)(u)&=&\frac{c''
\,\gam_{n,n-k}(k-\lam)}{\gam_{m,m-k}(k-\lam)} \int_0^{\pi/2
}\sin^{n-m-1}\psi \, \cos^{m-\lam-1}\psi \,d\psi\nonumber\\&\times&
\int_{S^{n-1} \cap
\eta^\perp} \!\!\!\!f(u\cos\, \psi \!+\!v\sin\, \psi)\, dv\nonumber\\
&=&\!\!\frac{\pi^{(m-n)/2}\, \sig_{n-m}}{2}\,\int_{S^{n-1} \cap
(\eta^\perp \oplus \bbr u )} \!\!\!\!f (w) |u\cdot w|^{m-\lam-1}\,
dw.\nonumber\eea Formula (\ref{yab2}) follows from (\ref{yab1}) by
(\ref{lim}).
\end{proof}
\begin{theorem}\label{restrm} Let $f\!\in \!D_e(S^{n-1})$, $\eta \!\in\! G_{n,m}$,
$1\!<\!m\!<\!n$. Suppose that $ f\!=\!M^{1-\lam} g$, where $Re
\,\lam <m$, $\lam \neq 0, -2, -4, \ldots\,$. Then the restriction of
$f$ onto $\eta$ is represented as $f=M^{1-\lam}_{S^{n-1} \cap \eta}
T_\eta ^\lam g$, where $T_\eta ^\lam$ has the form (\ref{teta}) and
$M^{1-\lam}_{S^{n-1} \cap \eta}$ denotes the same operator
$M^{1-\lam}$, but on the sphere $S^{n-1} \cap \eta$.
\end{theorem}
\begin{proof} For $Re \,\lam <1$, the statement is a particular case
of Theorem \ref{restrf} (set $k=1$). For other values of $\lam$, the
result follows by analytic continuation.
\end{proof}
\begin{remark}\label{neho} The restriction $\lam \neq 0, -2, -4,
\ldots\,$ in Theorems \ref{restrf} and \ref{restrm} is caused by the
Gamma function $\Gam (\lam/2)$ in the numerator of the corresponding
normalizing factor. It is evident from the proof, that both theorems
remain true also for $\lam =  -2\ell, \ell \in \bbn$, if we remove
 the normalizing factor. Then  $M^{1-\lam}$ in Theorem \ref{restrm}
 will be substituted for
$\tilde M^{1+2\ell}$; see (\ref{afu}).
\end{remark}

We will need the following generalization of Theorem \ref{restrm}.
\begin{theorem}\label{restrmi} Let
$f\!\in \!C_e(S^{n-1})$, $\mu\in \M_{e+}(S^{n-1})$, and let $\eta
\!\in\! G_{n,m}$, $1\!<\!m\!<\!n$. Suppose that $ f\!=\!M^{1-\lam}
\mu$, if $\lam <m$, $\lam \neq -2\ell, \ell \in \bbn$, and $
f\!=\!\tilde M^{1+2\ell} \mu$, if $\lam = -2\ell$.

\noindent {\rm (i)} There is a measure $\nu\in \M_{e+}(S^{n-1} \cap
\eta)$ such that
 the restriction of $f$ onto $S^{n-1} \cap \eta$ is represented as
$f=M^{1-\lam}_{S^{n-1} \cap \eta} \nu$.

\noindent {\rm (ii)} If $d\mu (\theta)=g(\theta)d\theta$, $g \in
C_e(S^{n-1})$, then {\rm (i)} holds with $d\nu (\theta)=(T_\eta
^\lam g)(\theta)d\theta$, where $T_\eta ^\lam g$ has the form
(\ref{teta}).

\noindent {\rm (iii)} If $\lam =  -2\ell, \ell \in \bbn$, then {\rm
(i)} and {\rm (ii)} hold with $ M^{1-\lam}_{S^{n-1} \cap \eta}$
substituted for  $\tilde M^{1+2\ell}_{S^{n-1} \cap \eta}$.
\end{theorem}
\begin{proof} STEP 1. Let first $\lam <m$, $\lam \neq 0, -2, -4, \ldots\,$.
We invoke the Poisson integral (\ref{pu}) so that
$$\Pi_t f=\Pi_t M^{1-\lam} \mu=M^{1-\lam} g_t, \qquad g_t=\Pi_t
\mu\in \D_e(S^{n-1}), \quad t\in (0,1).$$ Since $f$ is continuous,
then $\Pi_t f $ converges to $f$ as $t \to 0$ uniformly on
$S^{n-1}$, and therefore, uniformly on $S^{n-1} \cap \eta$. Hence,
for any test function  $\om \in \D(S^{n-1} \cap \eta)$, owing to
Theorem \ref{restrm}, we have \bea (f, \om)&=&\lim\limits_{t \to 0}
(\Pi_t f, \om)=\lim\limits_{t
\to 0} (M^{1-\lam} g_t, \om)\nonumber\\
&=&\lim\limits_{t \to 0} (M^{1-\lam}_{S^{n-1} \cap \eta} T_\eta
^\lam g_t,\om)=\lim\limits_{t \to 0} (
T_\eta ^\lam g_t, M^{1-\lam}_{S^{n-1} \cap \eta}\om)\nonumber\\
&=&\label{pora}\lim\limits_{t \to 0} (\nu_t,M^{1-\lam}_{S^{n-1} \cap
\eta}\om), \qquad \nu_t=T_\eta ^\lam g_t. \eea Thus, $\lim\limits_{t
\to 0} (\nu_t,M^{1-\lam}_{S^{n-1} \cap \eta}\om)$ exists for every
$\om \!\in  \!\D(S^{n-1}  \!\cap  \!\eta)$. If $\om$ is even, i.e.,
$\om \in \D_e(S^{n-1} \cap \eta)$, then, by Lemma \ref{l1}, we can
replace $\om$ by $M^{1-m+\lam}_{S^{n-1} \cap \eta}\om$ and conclude
that the limit $\lim\limits_{t \to 0} (\nu_t, \om)$ is well-defined
for every $\om \in \D_e(S^{n-1} \cap \eta)$. Since $\nu_t=T_\eta
^\lam \Pi_t \mu$ is an even function and the generic  test function
$\om\in \D(S^{n-1} \cap \eta)$ can be represented as $\om_+ +
\om_-$, where $\om_{\pm}$ are even and odd, respectively, it follows
that the limit $\lim\limits_{t \to 0} (\nu_t, \om)=\lim\limits_{t
\to 0} (\nu_t, \om_+)$ is well-defined for every $\om \in \D(S^{n-1}
\cap \eta)$ (not only for even $\om$, as stated above).
 Since
$\D'(S^{n-1}\cap \eta)$ is weakly complete, there is an even
distribution $\nu$ in $\D'(S^{n-1}\cap \eta)$ so that
$$(\nu,\om)=\lim\limits_{t \to 0} (\nu_t, \om), \qquad \om\in \D(S^{n-1} \cap \eta).$$ Furthermore, since
$(\nu_t, \om)=(T_\eta ^\lam \Pi_t \mu, \om)$ is non-negative for
every  non-negative $\om \in \D(S^{n-1} \cap \eta)$ and every $t\in
(0,1)$, then $\nu$ is a positive distribution and, by Theorem \ref
{sm}, $\nu$ is a measure in $\M_{e+}(S^{n-1}\cap \eta)$. Thus, by
(\ref {pora}),  $ (f, \om)=\lim\limits_{t \to 0}
(\nu_t,M^{1-\lam}_{S^{n-1} \cap \eta}\om)=(\nu,M^{1-\lam}_{S^{n-1}
\cap \eta}\om)$, which means that $f=M^{1-\lam}_{S^{n-1} \cap
\eta}\nu$, as desired.

If $d\mu (\theta)=g(\theta)d\theta$, $g \in C_e(S^{n-1})$, then
$\nu_t=T_\eta ^\lam \Pi_t g$ tends to $ T_\eta ^\lam
 g$ uniformly on $S^{n-1} \cap \eta$ as $t\to 0$. Hence, by
(\ref {pora}),  $(f, \om)= (T_\eta ^\lam  g,M^{1-\lam}_{S^{n-1} \cap
\eta}\om)$, which means $f=M^{1-\lam}_{S^{n-1} \cap \eta}T_\eta
^\lam g$.

STEP 2. Consider the  case $\lam =  -2\ell, \, \ell \in \bbn$, when
$f=\tilde M^{1+2\ell} \mu$, $\mu\in \M_{e+}(S^{n-1})$, and the
operator $T_\eta^{\lam} = T_\eta^{-2\ell}$ has the form
 $$(T_\eta^{-2\ell} h)(u)=\tilde
c\!\!\!\intl_{S^{n-1} \cap (\eta^\perp \oplus \bbr u )} \!\!\!
|u\cdot w|^{m+2\ell-1}\, h(w)\,dw,$$ cf. (\ref {teta}). For any
functions $h\in C(S^{n-1})$ and $\om \in C(S^{n-1} \cap \eta)$,
\be\label{fno}(T_\eta^{-2\ell}h, \om)=(h,{\overset *
T}{}^{-2\ell}_\eta  \om),\ee where
$$({\overset * T}{}^{-2\ell}_\eta \om)(\theta)\!=\!\frac{\Gam
(m/2)}{2\Gam (n/2)}\, \om \Big (\frac {{\rm Pr}_{\eta}\theta }{|{\rm
Pr}_{\eta}\theta |}\Big )\,|{\rm Pr}_{\eta}\theta |^{2\ell} \in C(S^{n-1}).$$ Indeed, using bi-spherical coordinates (see
(\ref{bsc})), we have \bea (T_\eta^{-2\ell}h, \om)&=&\tilde
c\!\!\!\intl_{S^{n-1} \cap \eta}\!\!\!\om (u)du\intl_{S^{n-1} \cap
(\eta^\perp \oplus \bbr u )}
\!\!\!h (w) |u\cdot w|^{m+2\ell-1}\, dw\nonumber\\
&=&\frac{\tilde c\,\sig_{n-m-1}}{\sig_{n-m}}\!\!\!\intl_{S^{n-1}
\cap \eta}\!\!\!\om
(u)du\intl_0^{\pi/2 }\sin^{n-m-1}\psi \, \cos^{m+2\ell-1}\psi \,d\psi\nonumber\\
&\times& \intl_{S^{n-1} \cap
\eta^\perp} \!\!\!\!h(u\cos\, \psi \!+\!v\sin\, \psi)\, dv\nonumber\\
&=&\frac{\tilde c\,\sig_{n-m-1}}{c''\,\sig_{n-m}}\intl_{S^{n-1}}
h(\theta)\, \om \Big (\frac {{\rm Pr}_{\eta}\theta }{|{\rm
Pr}_{\eta}\theta |}\Big )\,|{\rm Pr}_{\eta}\theta |^{2\ell}\,
d\theta=(h,{\overset * T}{}^{-2\ell}_\eta  \om).\nonumber\eea Let
$h=\Pi_t \mu$ and observe that the limit $\lim\limits_{t \to
0}(T_\eta^{-2\ell}\Pi_t \mu, \om)$ exists, because, by (\ref{fno}),
$ (T_\eta^{-2\ell}\Pi_t \mu, \om)=(\Pi_t \mu,{\overset *
T}{}^{-2\ell}_\eta  \om) \to (\mu,{\overset * T}{}^{-2\ell}_\eta
\om)$. Note that $(T_\eta^{-2\ell}\Pi_t \mu, \om)\ge 0$ for any
 non-negative $\om \in C(S^{n-1} \cap \eta)$. Applying the standard
 completeness argument (as in Step 1), we conclude, that there is a measure
$\nu \in \M_+(S^{n-1} \cap \eta)$ such that $$\lim\limits_{t \to
0}(T_\eta^{-2\ell}\Pi_t \mu, \om)=(\nu, \om)\quad \forall \om \in
C(S^{n-1} \cap \eta).$$ Using this equality, for $f=\tilde
M^{1+2\ell} \mu$ we obtain
 \bea (f, \om)&=&\lim\limits_{t \to 0}
(\Pi_t f, \om)=\lim\limits_{t \to 0} ( \Pi_t \tilde M^{1+2\ell} \mu,
\om)=\lim\limits_{t
\to 0} (\tilde M^{1+2\ell} \Pi_t \mu, \om)\nonumber\\
&& \text {\rm (use Theorem \ref{restrm} and Remark \ref{neho})}\nonumber\\
&=&\lim\limits_{t \to 0} (\tilde M^{1+2\ell}_{S^{n-1} \cap \eta}
T_\eta ^{-2\ell} \Pi_t \mu, \om)=\lim\limits_{t \to 0} (T_\eta
^{-2\ell} \Pi_t \mu, \tilde M^{1+2\ell}_{S^{n-1} \cap \eta}\om)\nonumber\\
&=&(\nu, \tilde M^{1+2\ell}_{S^{n-1} \cap \eta}\om).\nonumber\eea
This gives the result.

If $d\mu (\theta)=g(\theta)d\theta$, $g \in C_e(S^{n-1})$, then, by
Theorem \ref{restrm} and Remark \ref{neho}, for $\theta \in S^{n-1}
\cap \eta$ we have $$(\Pi_t f)(\theta)=(\Pi_t \tilde M^{1+2\ell}
g)(\theta)=(\tilde M^{1+2\ell}\Pi_t g)(\theta)= (\tilde
M^{1+2\ell}_{S^{n-1} \cap \eta}T_\eta ^{-2\ell} \Pi_t g)(\theta).$$
Owing to continuity of the operators $\tilde M^{1+2\ell}_{S^{n-1}
\cap \eta}$, $T_\eta ^{-2\ell}$, and $ \Pi_t$ in the relevant spaces
of continuous functions, by passing to the limit as $t \to 0$, we
obtain $f(\theta)=(\tilde M^{1+2\ell}_{S^{n-1} \cap \eta}T_\eta
^{-2\ell}  g)(\theta)$, $\theta \in S^{n-1} \cap \eta$, as desired.
\end{proof}

\section{Positive Definite Homogeneous Distributions}

 We remind some known facts; see, e.g.,
  \cite {GS}, \cite {Le}. Let $\S(\bbr^n)$ be the  Schwartz
 space of rapidly decreasing $C^\infty$-functions on $\bbr^n$
and $\S'(\bbr^n)$ its dual.  The Fourier transform of  $F \in
\S'(\bbr^n)$ is defined by
$$ \lng\hat F, \hat \phi\rng= (2\pi)^{n}\lng F, \phi\rng, \quad \hat \phi(y)=
\int_{\bbr^{n}} \phi(x) \, e^{ix \cdot y} \, dx, \quad \phi \in
\S(\bbr^n).$$ A distribution $F \in \S'(\bbr^n)$ is homogeneous of
degree $\lam \in \bbc$  if for any $\phi \in \S(\bbr^n)$ and any $a
>0$, $\lng F, \phi (x/a)\rng =a^{\lam +n}\,\lng F, \phi\rng$. Homogeneous
distributions on $\bbr^n$ are intimately connected with
distributions on $S^{n-1}$. Let first $f \in L^1(S^{n-1})$, $(E_\lam
f)(x)=|x|^\lam f (x/|x|)$, $x \in \bbr^n \setminus \{0\}$. The
operator $E_\lam$ generates a meromorphic $\S'$-distribution
$$ \lng E_\lam f, \phi \rng \!
=a.c. \int_0^\infty\! r^{\lam +n-1} u(r)dr, \quad
u(r)=\int_{S^{n-1}} f(\theta)\overline{\phi(r\theta)}d\theta,$$
where ``$a.c.$'' denotes analytic continuation in the
$\lam$-variable. The distribution $E_\lam f$ is regular if $Re
\,\lam >-n$ and admits simple poles at $\lam=-n, -n-1, \ldots $. The
above definition extends to all distributions $f\in\D'(S^{n-1})$ by
the formula
$$ \lng E_\lam f, \phi \rng=a.c.
\int_0^\infty r^{\lam +n-1} u(r)dr,  \quad u(r)=(f,
\phi(r\theta)),\footnote {\rm Here and on, different notations
$\lng \cdot, \cdot \rng$ and $(\cdot,\cdot)$ are used for distributions on $\bbr^n$ and $S^{n-1}$, respectively.}
$$ and the map $E_\lam : \D'(S^{n-1})\to \S'(\bbr^n)$
is  weakly continuous. If $f$ is orthogonal to all spherical
harmonics of degree $j$, then the derivative $u^{(j)}(r)$ equals
zero at $r=0$ and the pole at $\lam=-n-j$ is removable. In
particular, if  $f$ is an even distribution, i.e., $(f, \vp)=(f,
\vp^-), \; \vp^- (\theta)=\vp (-\theta) \quad \forall \vp \in \D
(S^{n-1})$, then the only possible poles  of $E_\lam f$ are $-n,
-n-2, -n-4, \dots $.

The Fourier transform of homogeneous distributions was extensively
studied by many authors; see \cite{Sa2} and references therein. We
restrict our consideration to even distributions, when  the operator
family $\{M^\a \}$ defined by (\ref{af}) naturally arises thanks to
the formula \be \label{cf} [E_{1-n-\a} f]^\wedge=2^{1-\a}
\pi^{n/2}\,E_{\a-1}M^\a f. \ee This formula amounts to Semyanistyi
\cite{Se}. If $f\in \D_e(S^{n-1})$, then (\ref{cf}) holds pointwise
for $0<Re \, \a <1$ (see, e.g., Lemma 3.3 in \cite{R3} ) and extends
in the $S'$-sense to all $\a\in \bbc$ satisfying \be\label{alf}\a
\notin \{1,3, 5, \ldots \}
 \cup \{1-n, -n-1, -n-3, \ldots \}.\ee
Since $\D_e(S^{n-1})$ is dense in $\D'_e(S^{n-1})$ and the maps
 $E_{1-n-\a}$ and $ E_{\a-1}$ are weakly
 continuous from $ \D'_e(S^{n-1})$ to $\S'(\bbr^n)$, then (\ref{cf}) extends to all $f \in \D'_e(S^{n-1})$.

Regarding the cases excluded in (\ref{alf}), we note that if
$\a=1+2\ell$ for some $\ell=0,1, \ldots$, then (\ref{cf}) is
meaningful if and only if $f$ is orthogonal to all spherical
harmonics of degree $2\ell$. If $\a=1-n-2\ell$ for some $\ell=0,1,
\ldots$, then, according to the spherical harmonic decomposition
$f=\sum_{j,k} f_{j,k} Y_{j,k}, \; j$ even, formula (\ref{cf}) is
substituted for the following: \bea \qquad &&[E_{2\ell} f]^\wedge
(\xi)=(2\pi)^n\sum_{j\le 2\ell} \sum_k f_{j,k} (-\Del)^{\ell
-j/2}Y_{j,k}(i\partial)\, \del (\xi)\\&{}&+2^{n+2\ell}
\pi^{n/2}\,E_{-n-2\ell}M^{1-n-2\ell}\Big [ f-\sum_{j\le 2\ell}\sum_k
 f_{j,k} Y_{j,k}\Big ](\xi), \nonumber\eea where $-\Del$ is the Laplace operator,
 $\partial=(\partial/\partial \xi_1, \ldots , \partial/\partial \xi_n
)$, and $\del (\xi)$ is the delta function. It is worth noting that
for $\a=1,3, 5, \ldots$, the distribution $[E_{1-n-\a} f]^\wedge$
can  also  be understood in the regularized sense without any
orthogonality assumptions. However, such regularization does not
preserve  homogeneity; see \cite{Sa0}, \cite{Sa2}.

Our main concern is positivity  and positive definiteness of even
homogeneous distributions.  The reader is referred to \cite{GV}  for
the general theory. A distribution $F\in \S'(\bbr^n)$ is {\it
positive} if $\lng F, \phi\rng \ge 0$ for all non-negative
  $\phi\in \S(\bbr^n)$. A similar definition holds for distributions on the sphere and on $\bbr^n \setminus \{0\}$.
 A distribution
 $F \in \S'(\bbr^n)$ is
 {\it positive definite} if $\hat F$ is positive.
For our  purposes, it is important to know,
 which even homogeneous distributions are
positive definite. Let us rewrite  (\ref{cf}) and (\ref{alf}) with
$1-n-\a$ replaced by $-\lam$. We have \be \label{cfl} [E_{-\lam}
f]^\wedge=2^{n-\lam} \pi^{n/2}\,E_{\lam -n}M^{1+\lam -n} f, \ee
\be\label{alfl}\lam\notin \Lam_0, \quad \Lam_0=\{n, n+2, n+4 \ldots
\}
 \cup \{0, -2, -4, \ldots \}.\ee
\begin{theorem}\label{pr0} Let $\lam\in \bbr \setminus  \Lam_0$,  $f \in \D'_e
(S^{n-1})$.

{\rm (i)} If $\lam< 0$ and $E_{-\lam} f$ is a positive definite
distribution, then $f=0$.

{\rm (ii)} For all   $\lam\in \bbr \setminus  \Lam_0$,  the
following statements are equivalent:

 {\rm (a)} $[E_{-\lam}f]^\wedge$ is a positive
distribution on $\rn \setminus \{0\}$ $($for $\lam> 0$, this can be
replaced by ``$E_{-\lam}f$ is a positive definite distribution on
$\rn$''$)$;

 {\rm (b)} $M^{1+\lam -n}f\in \M_{e+}(S^{n-1})$;

 {\rm (c)} $f=M^{1-\lam} \mu$ for some measure $\mu\in \M_{e+}(S^{n-1})$.

\noindent Furthermore, for any real $\lam \neq 0, -2, -4, \ldots$,
and any  $i =1,2, \ldots , n-1$, {\rm (c)} is equivalent to

{\rm (d)} $R_i f =R^{i-\lam}_{n-i, \perp} \mu$ for some measure
$\mu\in \M_{e+}(S^{n-1})$.
\end{theorem}
\begin{proof} {\rm (i)}
 Choose $\phi (x)=\exp (-|x|^m)\, p_{t,\theta}
(x/|x|)$, where $m \in 2\bbn$ and $p_{t,\theta} (\cdot)$ is the
Poisson kernel \be\label{pker} p_{t,\theta} (u)= \frac{1-t^2}{ (1-2t
u\cdot \theta +t^2)^{n/2}},\qquad  0<t <1; \quad u, \theta \in
S^{n-1}.\ee Then $\lng E_{\lam -n}M^{1+\lam -n} f, \phi\rng= c_\lam
(\Pi_t
 M^{1+\lam -n} f)(\theta)$, where
$$c_\lam= a.c. \int_0^\infty r^{\lam -1} \exp (-r^m)\,
dr=m^{-1}\Gam (\lam /m)$$ and $(\Pi_t M^{1+\lam -n} f)(\theta)$ is
the Poisson integral of $M^{1+\lam -n} f$. If  $E_{-\lam} f$ is a
positive definite distribution, then, by (\ref{cfl}), $ E_{\lam
-n}M^{1+\lam -n} f$ is a positive distribution. On the other hand,
if $\lam <0$ and $m>-\lam $, then $c_\lam<0$. Hence $\lng E_{\lam
-n}M^{1+\lam -n} f, \phi\rng$  can be  non-negative for {\it every}
non-negative $\phi\in \S(\bbr^n)$ only if $(\Pi_t M^{1+\lam -n}
f)(\theta)=0$ for every $0<t <1$ and $ \theta \in S^{n-1}$. The
latter implies $M^{1+\lam -n}f=0$, which is equivalent to $f=0$
because $M^{1+\lam -n}$ is injective; see Lemma \ref{l1}.

{\rm (ii)} Let $[E_{-\lam}f]^\wedge$ be a positive distribution on
$\rn \setminus \{0\}$. It means that for every  $\phi \in
\S(\bbr^n)$ such that $\phi \ge 0$ and $0\notin \supp \phi$, $\lng
[E_{-\lam}f]^\wedge, \phi \rng\ge 0$ or, by (\ref{cfl}), $\lng
E_{\lam -n}M^{1+\lam -n} f, \phi \rng\ge 0$. Choose $\phi (x)=\psi
(|x|) \om (x/|x|)$, where $\om\in \D (S^{n-1})$, $ \om \ge 0$, and
$\psi$ is a smooth non-negative function such that $\;\int_0^\infty
r^{\a +n-2} \psi (r) dr=1$ and $0\notin \supp \psi$. Then $$\lng
E_{\lam -n}M^{1+\lam -n} f, \phi \rng =(M^{1+\lam -n} f, \om)\ge
0,$$ and therefore, $M^{1+\lam -n} f\in \M_{e+}(S^{n-1})$; see
Theorem \ref {sm}.

 Conversely, let $\mu=M^{1+\lam -n} f\in \M_{e+}(S^{n-1})$  and let
 $\phi \in \S(\bbr^n)$; $\phi \ge 0$. In the case $\lam<0$ we
 additionally assume $0\notin \supp \phi$. By (\ref{cfl}), \bea
\lng [E_{-\lam} f]^\wedge, \phi\rng&=&2^{n-\lam} \pi^{n/2}\,\lng
E_{\lam -n}\mu,\phi\rng\nonumber\\&=&2^{n-\lam}
\pi^{n/2}\int_0^\infty r^{ \lam -1}dr\int_{S^{n-1}}\phi (r\theta)
d\mu(\theta)\ge 0.\nonumber\eea This proves equivalence of (a) and
(b). Equivalence of (b) and (c) follows from Lemma \ref{l1}.

Let us prove the equivalence of (c) and (d). If $R_i f
=R^{i-\lam}_{n-i, \perp} \mu$, $\mu\in \M_{e+}(S^{n-1})$, then, by
(\ref{cnd}),  \bea(f, R_i^*\vp)&=&(R_i f, \vp)=(R^{i-\lam}_{n-i,
\perp} \mu, \vp)= (\mu, \overset * R{}^{i-\lam}_{n-i} \vp^\perp
)\nonumber\\&=& c^{-1}(\mu,M^{1-\lam} R^*_i \vp), \qquad \vp \in
\D(G_{n,i}).\nonumber\eea Since any function $\om \in \D_e(S^{n-1})$
can be expressed as $\om=R_i^* \vp$ for some $\vp \in \D(G_{n,i})$
(see Lemma \ref{l34}), this gives $(f,\om)=
c^{-1}(\mu,M^{1-\lam},\om)$ which is (c). Conversely, let
$f=M^{1-\lam} \mu$,  $\mu\in \M_{e+}(S^{n-1})$, that is, $(f,\om)=
(\mu,M^{1-\lam},\om)$ for every $\om \in \D_e(S^{n-1})$. Choose
$\om=R_i^* \vp, \; \vp \in \D(G_{n,i})$. Then, as above, $(f,
R_i^*\vp)\!=\!(\mu,M^{1-\lam} R^*_i \vp)\!=\!c\,(\mu, \overset *
R{}^{i-\lam}_{n-i} \vp^\perp )$, which gives (d).
\end{proof}

 \section{$\lam$-intersection bodies}

  \subsection{Definitions and comments} We remind that  $\K^n $ is
 the set of all origin-symmetric star  bodies $K$ in $\bbr^n, \; n\ge
 2$;
 $\rho_K $ and $ ||\cdot||_K $ are the radial function  and  the Minkowski functional  of
$K$. The following definitions and statements are motivated by
Theorem \ref{pr0} and the previous consideration. Let $\lam $ be a
real number,
 \be\label{s55} s_\lam=\left\{
\begin{array}{ll}1 &  \mbox{\rm if $\lam>0, \quad \lam \neq  n, n+2, n+4, \ldots \,$}, \\
\Gam (\lam/2) &  \mbox{\rm if $\lam<0, \quad \lam \neq  -2, -4, \ldots\,$}.
\end{array}
\right.  \ee

The values $\lam=0, n, n+2, n+4, \ldots \,$ will not be considered
in the following, but values $\lam= -2, -4, \ldots\,$ will be
included. They become meaningful if we change normalization. For
$\lam \neq 0, n, n+2, n+4 \ldots \,$, let
 $\I^n_\lam $ be the set of bodies $K \in \K^n$, for which there is a
 measure $\mu\in \M_{e+}(S^{n-1})$  such that
$ s_\lam\rho_K=M^{1-\lam}\mu $ if $\lam \neq -2\ell, \; \ell \in
\bbn$, and $ \rho_K=\tilde M^{1-\lam}\mu \equiv \tilde
M^{1+2\ell}\mu $, otherwise.
 The equality $s_\lam\rho_K=M^{1-\lam}\mu$
means that for any $\vp \in \D(S^{n-1})$,
 $$
s_\lam\,\int_{S^{n-1}} \rho_K^{k}(\theta)\vp (\theta)\, d\theta= \int_{S^{n-1}}
 (M^{1-\lam}\vp)(\theta)\, d\mu (\theta),$$
where for $\lam \ge 1$, $(M^{1-\lam}\vp)(\theta)$ is understood in
the sense of analytic continuation. We remind the notation  $$
\Lam_0=\{n, n+2, n+4 \ldots \}
 \cup \{0, -2, -4, \ldots \}.$$
\begin{theorem}\label{aprr0} For $\lam \in \bbr \setminus \Lam_0$, the
following statements are equivalent:

{\rm (a)} $K \in \I^n_\lam$;

{\rm (b)} The Fourier transform $ [s_\lam
\,||\cdot||_K^{-\lam}]^\wedge $ is a positive  distribution on $\rn
\setminus \{0\}$ $($for $\lam> 0$, this can be replaced by
``$||\cdot||_K^{-\lam}$ is a positive definite distribution on
$\rn$''$)$;

{\rm (c)} $s_\lam \,M^{1+\lam -n}\rho_K^\lam\in \M_{e+}(S^{n-1})$;
\end{theorem}

The theorem is an immediate consequence of Theorem \ref{pr0} if the
latter is applied to $f=s_\lam \rho_K^\lam$. Another useful
characterization  is provided by Theorem \ref{pr0} (d).
\begin{theorem}\label{prr0} Let $\lam \in \bbr \setminus \Lam_0$.
If   $K \in\I^n_\lam$, then for every $i\in \{1,2, \ldots, n-1\}$ there is a measure
 $\mu\in \M_{e+}(S^{n-1})$ such that $s_\lam R_i \rho_K^\lam =R^{i-\lam}_{n-i,
\perp} \mu$. Conversely, if $$s_\lam R_i \rho_K^\lam
=R^{i-\lam}_{n-i, \perp} \mu$$ for some $i\in \{1,2, \ldots, n-1\}$
and some $\mu\in \M_{e+}(S^{n-1})$, then  $K \in\I^n_\lam$.
\end{theorem}

Although $\I^n_\lam $ was called ``the set of bodies'', the
definition of this set is purely analytic and extra work is needed
to understand  what {\it bodies}  (if any) actually constitute  the
class
 $\I^n_\lam$.

 The following  comments will be helpful.

{\bf 1.} The case $\lam >n$ is not so interesting, because by
Theorem \ref{aprr0}(c), $\I^n_\lam$ is either empty (if $\Gam
((n-\lam)/2)<0$) or coincides with the whole class $\K^n$ (if $\Gam
((n-\lam)/2)>0$).

{\bf 2.} The case  $\lam \in (0,n)$ agrees with the concept of
isometric embedding of the  space $(\bbr^n,||\cdot||_K)$ into
  $L_{-p}, \; p=\lam$; see Introduction. In the framework of this concept,
  all bodies $K \in
  \I^n_\lam$ can be regarded as ``unit balls of $n$-dimensional
  subspaces of $L_{-\lam}$''.

{\bf 3.} If $K \in
  \I^n_\lam$, where $\lam<0$ (one can replace $\lam$ by $=-p, \; p>0$), then
$$ ||u||_K^p= \int_{S^{n-1}}
 |\theta \cdot u|^{p} \,d\mu(\theta)
 $$
for some $\mu\in \M_{e+}(S^{n-1})$. This is the well known L\'evy
representation, characterizing isometric embedding of the  space
$(\bbr^n,||\cdot||_K)$ into $L_{p}$; see Lemma 6.4 in \cite{K3}.
 Statement (b) in Theorem \ref{aprr0} agrees with Theorem \ref{ttthm21}.
 Keeping this terminology, we can state the following
 \begin{proposition} Let $p>-n, \; p \neq 0$. Then $(\bbr^n,||\cdot||_K)$ embeds
 isometrically in $L_p$ if and only if  $K \in \I^n_{-p}$.
\end{proposition}

{\bf 4.} If $\lam =k\in \{1,2, \ldots, n-1\}$, then
$\I^n_\lam=\I^n_k$ coincides with the class of $k$-intersection
bodies; see Definition \ref {def45} and Theorem \ref{tthm21}.
Theorems \ref{aprr0} and \ref{prr0} provide new characterizations of
this class.

These comments   inspire the following
\begin{definition}\label{dffl}
Let $\lam<n, \; \lam \neq 0$. A body $K\in\K^n$ is said to be a
$\lam$-intersection body if $K \in \I^n_\lam$, or, in other words,
if there is a measure $\mu\in \M_{e+}(S^{n-1})$ such that $s_\lam
\rho_K^\lam=M^{1-\lam}\mu$ if $\lam \neq -2\ell, \; \ell \in \bbn$,
and $ \rho_K^{-2\ell}=\tilde M^{1+2\ell}\mu $, otherwise.
\end{definition}

The result of Theorem \ref{prr0} for $\lam =i=k$ can serve as an
alternative definition of $k$-intersection bodies in terms of Radon
transforms. This definition agrees with Definition \ref{kyy1} and
mimics Definition \ref{luu2}.
\begin{definition}\label{def4} Let $k\in \{1,2, \ldots, n-1\}$.
 A body $K\in\K^n$  is  a
$k$-intersection body if there is a non-negative measure $\mu$ on
$S^{n-1}$ such that \be\label{ib5} (R_{k} \rho_K^{k})(\xi)=(R_{n-k}
\mu)(\xi^\perp), \qquad  \xi \in G_{n,k}. \ee
\end{definition}

Equality (\ref{ib5}) is understood in the weak sense according
(\ref{fina}).  Namely, for $\vp \in C(G_{n,k})$ and $\vp^\perp
(\eta)= \vp (\eta^\perp), \; \eta \in G_{n,n-k}$, (\ref{ib5}) means
\be\label{neh} \int_{G_{n,k}} (R_{k} \rho_K^{k})(\xi) \vp (\xi) \,
d\xi=\int_{S^{n-1}} (R_{n-k}^*\vp^\perp)(\theta) \, d\mu
(\theta).\ee

\subsection{$\lam$-intersection
bodies of star bodies and closure in the radial metric}

As we mentioned in Introduction, the class of intersection bodies,
which coincides with $\I^n_\lam$ when  $\lam =1$, is the closure
 in the radial metric of the class of intersection bodies of star
bodies. Below we extend this result to all $\lam <n, \; \lam \neq
0$, in the framework of the unique approach. We remind (see
Definition \ref{kyy1}) that $K\in \K^n$ is a $k$-intersection body
of a body $L\in \K^n$ and write $K=\I\B_k (L)$ if \be\label{dive}
\vol_{k} (K\cap \xi)=\vol_{n-k} (L\cap \xi^\perp)\qquad \forall \xi
\in G_{n,k}.\ee
 Let  $\I\B_{k,n}$  be the set of all bodies $K\in \K^n$ satisfying
(\ref{dive}) for some $L\in \K^n$.

How can we extend the purely geometric property (\ref{dive}) to
non-integer values of  $k$? To this end, we first express
(\ref{dive}) in terms of the generalized cosine transforms
(\ref{af}).
\begin{lemma} If $K=\I\B_k (L)$ is infinitely smooth, then
\be\label{kl} \rho_L^{n-k}\!=\!c\,M^{1-n+k} \rho_K^{k}, \qquad
\rho_K^{k}\!=\!c^{-1}M^{1-k}\rho_L^{n-k}, \ee
$$c=\pi^{k-n/2}(n-k)/k.$$
\end{lemma}
\begin{proof} We make use of (\ref{conn}), where we set $i=k$, $\a=1-n+k$ and $f=\rho_K^{k}$.
By (\ref{lim}), this gives \be \label {bu1} R_k \rho_K^{k}=\tilde c
R_{n-k,\perp}M^{1-n+k} \rho_K^{k}, \qquad \tilde c=
\frac{\pi^{k-n/2}\, \sig_{n-k-1}}{\sig_{k-1}}.\ee On the other hand,
if  $ K =\I\B_k (L)$ is infinitely smooth, then, according to
(\ref{dive}) and the equality  \be\label{rraa}\vol_k(K\cap \xi)
=\frac{\sig_{k-1}}{k}\,(R_k \rho_K^k )(\xi), \ee  we have \be \label
{bu2} R_{k} \rho_K^{k}=\frac{k\,\sig_{n-k-1}}{(n-k)\,\sig_{k-1}
}\,R_{n-k,\perp} \rho_L^{n-k}.\ee Comparing (\ref{bu1}) and
(\ref{bu2}), owing to injectivity of the Radon transform, we obtain
the first equality in (\ref{kl}). The second equality
 follows from the first one by (\ref{st}).
\end{proof}

Equalities (\ref{kl}) are extendable to non-integer values of $k$.
We denote
$$
c_{\lam,n}=\pi^{\lam-n/2}(n\!-\!\lam)/\lam,  $$ and let $s_\lam$ be defined by (\ref{s55}).
\begin{definition}\label{df} Let $\lam <n, \; \lam \neq 0$;
 $\;K,L \in \K^n$. We say that $K$ is a $\lam$-intersection body
of  $L$ and  write $K=\I\B_\lam (L)$ if $s_\lam\rho_K^\lam
\!=\!c_{\lam,n}^{-1}M^{1-\lam} \rho_L^{n-\lam}$ in the case $\lam
\neq -2\ell, \; \ell \in \bbn$, and $\rho_K^{-2\ell} =\tilde
M^{1+2\ell} \rho_L^{n+2\ell}$, otherwise.
 The set of all
$\lam$-intersection bodies of star bodies will be denoted by
$\I\B_{\lam,n}$. We also denote \be \I\B_{\lam,n}^\infty\!=\!\{K \in
\I\B_{\lam,n}: \,\rho_K\in \D_e(S^{n-1})\}.\ee
 \end{definition}

By (\ref{st}),  equality $s_\lam\rho_K^\lam
\!=\!c_{\lam,n}^{-1}M^{1-\lam} \rho_L^{n-\lam}$  is equivalent to $
\rho_L^{n-\lam}=s_\lam\,c_{\lam,n} \,M^{1-n+\lam} \rho_K^{\lam}$.
Both equalities are generally understood in the sense of
distributions, for instance, $$s_\lam(\rho_K^\lam, \vp) =
c_{\lam,n}^{-1}(\rho_L^{n-\lam}, M^{1-\lam} \vp), \qquad \vp \in
\D(S^{n-1}).$$ If $K$ (or $L$) is smooth, then $s_\lam\rho_K^\lam
(\theta)\!=\!c_{\lam,n}^{-1}(M^{1-\lam} \rho_L^{n-\lam})(\theta)$
pointwise for every $\theta \!\in \!S^{n-1}$.

\begin{theorem}\label{tcl} Let $\lam <n, \; \lam \neq 0$. If $\lam
\neq -2\ell, \; \ell \in \bbn$, then the class $\I^n_\lam$ of
$\lam$-intersection bodies is the closure of the classes
$\I\B_{\lam,n}$ and $\I\B_{\lam,n}^\infty $ of $\lam$-intersection
bodies of star bodies in the radial metric:
 \be\label{cl2}
\I^n_\lam=\cl \, \I\B_{\lam,n}=\cl \,  \I\B_{\lam,n}^\infty.\ee
 If $\lam
= -2\ell, \; \ell \in \bbn$, then $\I^n_\lam \subset \cl \,
\I\B_{\lam,n}=\cl \,  \I\B_{\lam,n}^\infty$.
\end{theorem}
\begin{proof} STEP 1. We first prove that $\I^n_\lam\subset \cl \,
\I\B_{\lam,n}^\infty$. Let $K \in \I^n_\lam$, i.e.,

(a) $s_\lam\rho_K^\lam =M^{1-\lam} \mu, \;  \mu \in
\M_{e+}(S^{n-1})$, if $\lam \neq -2\ell, \; \ell \in \bbn$, and

(b) $ \rho_K^{-2\ell}=\tilde M^{1+2\ell}\mu $, otherwise.

\noindent Our aim is to define a sequence $K_j \in
\I\B_{\lam,n}^\infty$ such that $\rho_{K_j} \to \rho_{K}$ in the
$C$-norm.  Consider the Poisson integral $\Pi_t \rho_K^\lam$ (see
(\ref{pu})), that converges to $\rho_K^\lam$ in the $C$-norm when
$t\to 1$. In the case (a), for any test function $\om \in
\D(S^{n-1})$ we have
$$ (\Pi_t \rho_K^\lam,\om)=(\rho_K^\lam,\Pi_t \om)=s_\lam^{-1}(\mu,
M^{1-\lam}\Pi_t \om)=s_\lam^{-1}(M^{1-\lam}\Pi_t\mu, \om).$$ Similarly, in the
case (b),  we  have a pointwise equality $(\Pi_t
\rho_K^{-2\ell})(\theta)\!=\!(\tilde M^{1+2\ell}\Pi_t\mu)(\theta)$,
$ \theta \in S^{n-1}$. Choose  $K_j$ so that
 $\rho_{K_j}^\lam=\Pi_{t_j}\rho_K^\lam$, where $t_j$ is a
sequence in $(0, 1)$ approaching $1$. Clearly, $K_j $  converges to
 $K$ in the radial metric. Moreover,  $K_j \in
 \I\B_{\lam,n}^\infty$, because $\rho_{K_j}^\lam
=s_\lam^{-1}c_{\lam,n}^{-1}M^{1-\lam} \rho_{L_j}^{n-\lam}$ and
$\rho_{K_j}^{-2\ell} =\tilde
 M^{1+2\ell}\rho_{L_j}^{n+2\ell}$, where  the bodies $L_j$ are
 defined by $\rho_{L_j}^{n-\lam}= c_{\lam,n}\Pi_{t_j}\mu$ in the case (a), and
$\rho_{L_j}^{n+2\ell}= \Pi_{t_j}\mu$ in the case (b), respectively.

Conversely, let $K \in \cl \,  \I\B_{\lam,n}^\infty$, $\lam \neq -2,
-4, \ldots \,$.  It means that there is a sequence of $K_j \in
\I\B_{\lam,n}^\infty$ such that $\lim\limits_{j \to \infty} ||\rho_K
- \rho_{K_j}||_{C(S^{n-1})}=0$ and
$s_\lam\rho_{K_j}^\lam=c_{\lam,n}^{-1}M^{1-\lam}\rho_{L_j}^{n-\lam}$,
$ \rho_{L_j}\in \D_{e+}(S^{n-1})$. If $j \to \infty$, then for every
$\om\in\D(S^{n-1})$, \be\label{barb} s_\lam (\rho_{K_j}^\lam,
M^{1-n+\lam}\om) \to s_\lam (\rho_{K}^\lam,
M^{1-n+\lam}\om)\!=\!s_\lam ( M^{1-n+\lam}\rho_{K}^\lam, \om).\ee
The right-hand side of (\ref{barb}) is non-negative, because by
(\ref{st}), for every $j$ and every $\om \in \D_{e+}(S^{n-1})$,
$$s_\lam (\rho_{K_j}^\lam,
M^{1-n+\lam}\om)=c_{\lam,n}^{-1}(M^{1-\lam}\rho_{L_j}^{n-\lam},M^{1-n+\lam}\om)=
c_{\lam,n}^{-1}(\rho_{L_j}^{n-\lam},\om)\ge 0.$$ By Theorem \ref
{sm}, it follows that $s_\lam\,M^{1-n+\lam}\rho_{K}^\lam$ is a
non-negative measure. We denote it by $\mu$.
 By (\ref{st}), for any $\om \in
\D(S^{n-1})$,
$$ s_\lam (\rho_{K}^\lam,
\om)=s_\lam (M^{1-n+\lam}\rho_{K}^\lam, M^{1-\lam}
\om)=(\mu,M^{1-\lam} \om)=(M^{1-\lam}\mu, \om),$$ i.e., $K \!\in \!
\I^n_\lam$. This gives $\I\B_{\lam,n}^\infty \!\subset \!\I^n_\lam$
and, by above, $\I^n_\lam\!=\!\cl \, \I\B_{\lam,n}^\infty$.

STEP 2. It remains to prove that $\cl \, \I\B_{\lam,n}^\infty =\cl
\, \I\B_{\lam,n}$. Since $\I\B_{\lam,n}^\infty
\subset\I\B_{\lam,n}$, then $\cl \, \I\B_{\lam,n}^\infty \subset \cl
\, \I\B_{\lam,n}$. To prove the opposite inclusion, let $K \in \cl
\, \I\B_{\lam,n}$ and consider the case $\lam \neq -2, -4, \ldots
\,$.  We have to show that there is a sequence of smooth bodies
$K_j$, which converges to $K$ in the radial metric and such that
$s_\lam
 \rho_{K_j}^\lam=c_{\lam,n}^{-1}M^{1-\lam}\rho_{L_j}^{n-\lam}$ for
some bodies $L_j \in \K^n$. Since $K \in \cl \, \I\B_{\lam,n}$,
there is a sequence  $\tilde K_j\in \K^n$ such that $\lim\limits_{j
\to \infty} ||\rho_{\tilde K_j}-\rho_{K}||_{C(S^{n-1})}=0$ and
$s_\lam \rho_{\tilde
 K_j}^\lam=c_{\lam,n}^{-1}M^{1-\lam}\rho_{\tilde L_j}^{n-\lam}$ for some
bodies $\tilde L_j \in \K^n$. We define smooth bodies $K_j$ and
$L_j$ by
$$ \rho_{K_j}^\lam=\Pi_{1-1/j}\rho_{\tilde
 K_j}^\lam, \qquad   \rho_{L_j}^{n-\lam} =\Pi_{1-1/j}  \rho_{\tilde L_j}^{n-\lam},$$
 where $\Pi_{1-1/j}$ stands for the Poisson integral with
 parameter $1-1/j$. Since operators $\Pi_{1-1/j}$ and $M^{1-\lam}$
 commute, then $s_\lam
\rho_{ K_j}^\lam\!=\!c_{\lam,n}^{-1}M^{1-\lam}\rho_{L_j}^{n-\lam}$,
and therefore, $K_j \in \I\B_{\lam,n}^\infty$. On the other hand, by
the properties of the Poisson integral \cite{SW}, $$|\rho_{
K_j}^\lam-\rho_{K}^\lam |\le |\Pi_{1-1/j}\rho_{\tilde K_j}^\lam -
\Pi_{1-1/j}\rho_{K}^\lam |+ |\Pi_{1-1/j}\rho_{K}^\lam
-\rho_{K}^\lam| \to 0$$ as $j \to \infty$. It means, that $K \in \cl
\, \I\B_{\lam,n}^\infty$ or $\cl \, \I\B_{\lam,n} \subset \cl \,
\I\B_{\lam,n}^\infty$. Hence, by above,  $\cl \, \I\B_{\lam,n} = \cl
\, \I\B_{\lam,n}^\infty$. For $\lam = -2, -4, \ldots \,$, the
argument is similar.
\end{proof}
\begin{remark} If $\lam = -2, -4, \ldots \,$, we cannot prove the
coincidence of $\I^n_\lam $ and $\cl \,  \I\B_{\lam,n}^\infty$,
because the proof of the embedding $\cl \,  \I\B_{\lam,n}^\infty
\subset \I^n_\lam $ relies heavily on the fact that  $M^{1-\lam}$ is
an isomorphism of  $\D_e (S^{n-1})$. If $\lam = -2, -4, \ldots \,$,
this is not so, and  the operator $\tilde M^{1-\lam}$ has a
nontrivial kernel, which consists of spherical harmonics of degree $
>2\ell$; see \cite {R3} for  details.
\end{remark}

It is interesting to translate Theorem \ref{tcl} for $\lam=-p, \;
p>0$, into the language of isometric embeddings. Ignoring a
non-important positive constant factor and using polar coordinates,
one can replace the equalities $s_\lam\rho_K^\lam
\!=\!c_{\lam,n}^{-1}M^{1-\lam} \rho_L^{n-\lam}$ and $\rho_K^{-2\ell}
=\tilde M^{1+2\ell} \rho_L^{n+2\ell}$ in Definition \ref {df} by
\be\label{moby} ||u||_{K}^p=\int_{L} |x \cdot u|^{p} \,dx, \qquad
u\in
 S^{n-1}.\ee
\begin{corollary} ${}$

\noindent {\rm (i)} A unit ball of every
 $n$-dimensional subspace of $L_p$, can be approximated in  the radial metric by  bodies $K$, defined
by (\ref{moby}), where  $L\in \K^n$ has a $C^\infty$ boundary.

\noindent {\rm (ii)} If, moreover,  $p \neq 2,4, \ldots\,$, then the
set of unit balls of all $n$-dimensional subspaces of $L_p$, can be
identified with the closure in the radial metric of the set of
bodies $K$ satisfying (\ref{moby}) for some smooth body $L\in \K^n$
$($one can also consider arbitrary bodies $L\in \K^n)$.
\end{corollary}

\subsection{Central sections of $\lam$-intersection
bodies}

It is known, that a cross-section $K\cap \eta$ of a body $K \in
I^n_k$ by any $m$-dimensional central plane $\eta$ is a
$k$-intersection body  in $\eta$ provided $1\le k<m<n$. This fact
was established in \cite[Proposition 3.17]{Mi1} by using Theorem
\ref{tthm21} and a certain  approximation procedure. Below we
present more general results, including sections of $k$-intersection
bodies of star bodies and the case of non-integer $k=\lam$. These
results are consequences of the restriction theorems from Section
3.3.
\begin{theorem}\label{tyi} Let $1\le
k<m<n$, $\eta\in G_{n,m}$. If $K=\I\B_k (L)$ in $\bbr^n$, then
$K\cap \eta =\I\B_k (\tilde L)$ in $\eta$, where the body $\tilde L
$ is defined by \be \rho_{\tilde L}^{m-k}(u)=c_{k,m,n}\intl_{S^{n-1}
\cap (\eta^\perp \oplus \bbr u )} \!\!\!\rho_{L}^{n-k} (w) |u\cdot
w|^{m-k-1}\, dw, \ee
$$
u \in S^{n-1} \cap \eta, \qquad c_{k,m,n}=\frac{(m-k)\,
\sig_{n-m}}{2(n-k)}\,. $$
\end{theorem}
\begin{proof} By (\ref{rraa}) and (\ref{yab2}) (with
$f=\rho_{L}^{n-k}$),
 \bea \vol_{k} (K\cap \xi)&=&\vol_{n-k}
(L\cap \xi^\perp)=
\frac{\sig_{n-k-1}}{n-k}(R_{n-k}\rho_{L}^{n-k})(\xi^\perp)\nonumber\\\label{mu}&=&
\frac{c\,\sig_{n-k-1}}{n-k}(R_{m-k}T_\eta^k\rho_{L}^{n-k})(\xi^\perp\cap
\eta)\\&=& \frac{\sig_{m-k-1}}{m-k}\, (R_{m-k}\rho_{\tilde
L}^{m-k})(\xi^\perp\cap \eta)=\vol_{m-k} (\tilde L\cap
\xi^\perp),\nonumber\eea as desired.
\end{proof}

Theorem \ref{tyi} has the following generalization.
\begin{theorem}\label{tyif} Let $1<m<n$, $\eta\in G_{n,m}$ and suppose that $\lam<m$, $\lam \neq 0$. If $K=\I\B_\lam (L)$ in $\bbr^n$, then
$K\cap \eta =\I\B_\lam (\tilde L)$ in $\eta$, where  the body
$\tilde L $ is defined by \be\label{yoav} \rho_{\tilde
L}^{m-\lam}(u)=\tilde
 c\intl_{S^{n-1} \cap (\eta^\perp \oplus \bbr u )}
\!\!\!\rho_{L}^{n-\lam} (w) |u\cdot w|^{m-\lam-1}\, dw, \ee
$$
u \in S^{n-1} \cap \eta, \qquad \tilde c=\left\{
\begin{array}{ll}\displaystyle{\frac{(m-\lam)\, \sig_{n-m}}{2(n-\lam)}} &
\mbox{\rm if $\lam \neq -2\ell, \; \ell \in
\bbn$}, \\
{}\\
\pi^{(m-n)/2}\, \sig_{n-m}/2 &  \mbox{\rm otherwise}.
\end{array}
\right.
$$ Moreover, if $K\in
\I^n_\lam$ in $\bbr^n$, then $K\cap \eta \in \I^m_\lam$ in $\eta$.
\end{theorem}
\begin{proof} Let $\lam \!\neq\! -2\ell, \, \ell \!\in\!
\bbn$, and let $\theta \!\in \!S^{n-1} \cap \eta$. By Definition
\ref{df}, $s_\lam\rho_{K}^{\lam}=c_{\lam, n}^{-1}
M^{1-\lam}\rho_{L}^{n-\lam}$, and Theorem \ref{restrmi} (with
$f=s_\lam\rho_{K}^{\lam}$ and $g=c_{\lam, n}^{-1}
\rho_{L}^{n-\lam}$) yields
$$ s_\lam\rho_{K}^{\lam} (\theta)=(M^{1-\lam}_{S^{n-1} \cap \eta}T_\eta ^\lam
[c_{\lam, n}^{-1} \rho_{L}^{n-\lam}]) (\theta)=
c_{\lam, m}^{-1}(M^{1-\lam}_{S^{n-1} \cap \eta} \rho_{\tilde L}^{m-\lam}) (\theta),
$$ where  $\rho_{\tilde L}^{m-\lam}=c\,T_\eta ^\lam \rho_{L}^{n-\lam}$, $c=\pi^{(n-m)/2} (m-\lam)/(n-\lam)$.
 By Definition \ref{df} and (\ref{teta}), we are done. If $\lam =
-2\ell, \, \ell \in \bbn$, then, as above,
$$\rho_{K}^{-2\ell} (\theta)=(\tilde M^{1+2\ell}_{S^{n-1} \cap \eta}T_\eta ^{-2\ell}
\rho_{L}^{n+2\ell}) (\theta)=
(M^{1-\lam}_{S^{n-1} \cap \eta} \rho_{\tilde L}^{m+2\ell}) (\theta)
$$ where  $\rho_{\tilde L}^{m+2\ell}=T_\eta ^{-2\ell} \rho_{L}^{n+2\ell}$. This gives (\ref{yoav}).

Furthermore, if $K\in \I^n_\lam$,  $\lam \neq -2\ell, \; \ell \in
\bbn$, then, by Definition \ref{dffl}, $s_\lam\rho_{K}^{\lam}=
M^{1-\lam}\mu$, $\mu\in \M_{e+}(S^{n-1})$. Hence, by Theorem
\ref{restrmi}, there is a measure $\nu\in \M_{e+}(S^{n-1} \cap
\eta)$ such that
 the restriction of $s_\lam\rho_{K}^{\lam}$ onto $S^{n-1} \cap \eta$ is represented as
$s_\lam\rho_{K}^{\lam}=M^{1-\lam}_{S^{n-1} \cap \eta} \nu$. It means
that $K\cap \eta \in \I^m_\lam$ in $\eta$. In the case $\lam =
-2\ell, \, \ell \in \bbn$, the argument is similar. \end{proof}

\section{Examples of $\lam$-intersection
bodies}

The definition of the classes $\I^n_\lam$ and $\I\B_{\lam,n}$ and
all known characterizations   are purely analytic. Unlike the case
$\lam=1$, when an intersection body of a star body is explicitly
defined by a simple geometric procedure, it is not clear how can we
construct  $\lam -$intersection bodies
 in the general case.  Below we give some examples, when the radial
function of a $\lam -$intersection body can be
 explicitly determined. These examples utilize the
 generalized cosine transforms.
\begin{example}\label{nin0} Let $\lam <1, \; \lam \neq 0$. This case is the simplest. Indeed,
given a non-negative measure $\mu$ on $S^{n-1}$,  the relevant $\lam
-$intersection body can be directly constructed by the formula
$\rho_K^\lam =M^{1-\lam} \mu$,  if $\lam \neq -2\ell, \ell \in
\bbn$, and $\rho_K^{-2\ell} =\tilde M^{1+2\ell}\mu$, otherwise. In
other words (cf. (\ref{afu})), \be \rho_K^\lam (u) =\int_{S^{n-1}}
|\theta \cdot u|^{-\lam} \,d\mu (\theta). \ee This fact (with $\lam$
replaced by $-p$) is a reformulation of Theorem 6.17 from \cite{K3},
which was stated in the language of isometric embeddings and relies
on the P. L\'evy characterization; see also Lemma 6.4 and Theorem
4.11 in \cite{K3}.
\end{example}
\begin{example}\label{ninn} If $n-3\le \lam<n, \; \lam>0$, then
 $\I^n_\lam$ includes
all origin-symmetric convex bodies in $\bbr^n$.
\end{example}
This fact is due to Koldobsky \cite[Corollary 4.9]{K3}. It can be
proved using a slight modification of the argument from \cite [Sec.
7]{R2} as follows. By Theorem \ref{aprr0} (c), it suffices to check
that for any o.s. convex body $K$, $M^{1+\lam -n}\rho_K^\lam\in
\M_{e+}(S^{n-1})$. For $\lam
 \ge n-1$, this is obvious. To handle the case $n-3\le \lam
 <n-1$, suppose first
 that $K$ is infinitely smooth. Using polar coordinates, for $ Re\, \a>0$, we can write \be\label{gnd}
 (M^{\a}\rho_K^{\a+n-1})(u)=(\a+n-1)\,\gam_n(\a)\int_K |x\cdot u|^{\a
 -1}\, dx.\ee
 Then $M^{1+\lam -n}\rho_K^\lam$ can be realized as analytic
 continuation ($a.c.$) at $\a=1+\lam -n$ of the right-hand side of
 (\ref{gnd}). The latter can be written as $$
I(\a)=2(\a+n-1)\gam_n(\a)\int_0^\infty t^{\a -1} A_{K,u} (t) \,dt,$$
$ A_{K,u} (t)=\vol_{n-1}(K\cap\{tu +u^\perp\}).$ Taking analytic
 continuation (see
\cite[Chapter 1]{GS}), for $-2<\a<0$ (which is equivalent to $n-3\le
\lam <n-1$) we get
$$a.c.I(\a)=c_1\int_0^\infty t^{\a -1} [A_{K,u} (t) - A_{K,u}
(0)]\,dt.
$$ Similarly, $a.c.I(\a)$ at $\a=-2$ (which corresponds to $\lam=n-3$)
is $c_2 A''_{K,u} (0)$. Following \cite{GS}, one can easily check
that constants $c_1$ and $c_2$ are negative.
 Since $K$ is convex, both analytic continuations are
positive, and therefore $M^{1+\lam -n}\rho_K^\lam>0$. If $K$ is an
arbitrary o.s. convex body, we approximate it in the radial metric
by smooth o.s. convex bodies $K_j$. Then for any test function $\om
\in \D_{+}(S^{n-1})$, by the previous step we have \bea (M^{1+\lam
-n}\rho_K^\lam, \om)&=&(\rho_K^\lam, M^{1+\lam
-n}\om)=\lim\limits_{j \to \infty}
(\rho_{K_j}^\lam, M^{1+\lam -n}\om)\nonumber\\
&=&\lim\limits_{j \to \infty} ( M^{1+\lam -n}\rho_{K_j}^\lam,\om)\ge
0.\nonumber\eea  Hence, by Theorem \ref {sm}, $M^{1+\lam
-n}\rho_K^\lam$ is a non-negative measure and the proof is complete.

\begin{example}\label{zhan}  If $\rho_K^\lam=
{\overset *  R}{}_{n-i}^{i-\lam}\nu$ for some $\nu \in \M_+ (G_{n,
n-i})$ and $\lam \le i<n$, then $K \in \I^n_\lam$.

Indeed, for any test function $\om \in \D(S^{n-1})$, by (\ref{con})
(with $\a=1-\lam$) we have \bea (\rho_K^\lam, \om)&=&({\overset *
R}{}_{n-i}^{i-\lam}\nu,\om)=(\nu,R_{n-i}^{i-\lam}\om)=(\nu^\perp,R_{n-i,
\perp}^{i-\lam}\om)\nonumber\\&=&c^{-1}(\nu^\perp, R_i M^{1-\lam}
\om)=c^{-1}( R_i^*\nu^\perp, M^{1-\lam} \om), \quad c=
\frac{2\pi^{(i-1)/2}}{\sig_{i-1}}.\nonumber\eea It means that for
$0<\lam \le i<n$ and $\nu \in \M_+ (G_{n, n-i})$,\be \rho_K^\lam=
{\overset * R}{}_{n-i}^{i-\lam}\nu \quad \Longleftrightarrow \quad
\{\rho_K^\lam=M^{1-\lam}\mu, \quad \mu=c^{-1}R_i^*\nu^\perp\}.\ee By
Definition \ref{dffl}, this gives the result. The particular case
$\lam = i$ implies the embedding into $\I^n_i$ of the Zhang's class
$\Z^n_{n-i}$; see Definition \ref{zyy1}. This embedding was proved
in \cite {K4} and \cite{Mi1}  in a different way; see also
\cite{Mi2}, where it is proved that $\Z^n_{n-i}$ is a proper subset
of $\I^n_i$ if $2\le i\le n -2$.
\end{example}

\begin{example} If $0<(i-1)/2 <\lam \le i<n$ and
$\rho_K^\lam=M^{i-\lam}\mu$ for some $\mu \in \M_+ (S^{n-1})$,  then
$K \in\I^n_\lam$.

Indeed, by Lemma \ref{lcon} (with $\a=i-\lam, \; \b=1-\lam$), $
\rho_K^\lam=M^{i-\lam}\mu= M^{1-\lam}A_{i-\lam,1-\lam }$, where
$A_{i-\lam,1-\lam }$ is an integral operator which preserves
positivity provided $i-\lam>1-\lam>1-n,\; (i-\lam)+(1-\lam)<2$. This
is just our case.
\end{example}
\begin{example} One can construct bodies $K \in\I^n_\lam$ from bodies
$L \in\I^n_\del$ by the formula $\rho_K=\rho_L^{\lam/\del}$ provided
$0<\del<\lam< n$.

Indeed,  by Definition \ref{dffl}, there is a measure $\mu \in \M_+
(S^{n-1})$ so that $ \rho_L^{\del}=M^{1-\del}\mu$. Then, by
 Lemma \ref{lcon} (with $\a=1-\del, \; \b=1-\lam$),  $
\rho_K^\lam=\rho_L^{\del}=M^{1-\del}\mu= M^{1-\lam}A_{1-\del,1-\lam
}\mu$, and we are done. This example generalizes the corresponding
result from \cite[p. 533, Statement (c)]{Mi1}, which
 was obtained in a different way for the case, when
 $\lam$ and $\del$ are integers.
\end{example}

\begin{example}\label {ngiv} Let
\be B^n_q=\{ x \in \bbr^n\,: ||x||_q=\Big (\sum_{k=1}^n |x_k|^q \Big
)^{1/q} \le 1\}.\ee
 If $0<q\le 2$, then $B^n_q \in\I^n_\lam$ for all
$\lam \in (0,n)$. If $2<q<\infty$, $\lam \in (0,n)$, then $B^n_q
\in\I^n_\lam$ if and only if $\lam \ge n-3$.

Both statements are due to Koldobsky. The first one follows from the
fact that for $0<q\le 2$
 the  Fourier transform of $||x||_q^{-\lam}$ is a positive
 $\S'$-distribution (see Lemmas 3.6 and 2.27 in \cite{K3}).
The second statement  is a reformulation of Theorem 4.13 from
\cite{K3}. The  ``if'' part is a consequence of Example \ref{ninn}.
\end{example}

\section {$(q,\ell)$-balls}

 In this section we consider one more example, which resembles Example \ref
 {ngiv}, but does not fall into its scope and requires a separate
 consideration. Let $$ x =(x', x'')\in \rn, \quad x' \in
\rnl=\overset{n-\ell}{\underset{j=1}{\oplus}}\,\bbr e_j,    \quad
x''\in \rl=\overset{n}{\underset{j=n-\ell +1}{\oplus}}\,\bbr e_j,$$
$e_1, \ldots, \e_n$ being coordinate unit vectors. Consider the
$(q,\ell)$-ball \be\label{ball} B^n_{q,\ell}=\{x:
||x||_{q,\ell}=(|x'|^q +|x''|^q)^{1/q} \le 1\}, \qquad q>0.\ee We
wonder, for which triples $(q,\ell,n)$, $B^n_{q,\ell}$ is a
$\lam$-intersection body. To study this problem, we need some
preparation. Consider the Fourier integral
 \be\label{gaql} \gam_{q,\ell}(\eta)=\int_{\rl}e^{-|y|^q} e^{i y\cdot
\eta}\, dy, \qquad \eta \in \rl, \qquad q>0.\ee The function
$\gam_{q,\ell}(\eta)$ is uniformly continuous on $\rl$ and vanishes
at infinity.
\begin{lemma}\label{pdd0} If $0<q\le 2$, then $\gam_{q,\ell}(\eta)>0$ for all $\eta \in
\rl$.
\end{lemma}
\begin{proof} (Cf. \cite [p. 44, for $\ell=1$]{K3}). For $\eta=0$,
the statement is obvious. It is known (see, e.g., \cite{SW}), that
\be\label{75}
[e^{-t|\cdot\,|^2}]^{\wedge}(\eta)=\pi^{\ell/2}t^{-\ell/2}e^{-|\eta|^2/4t},
\qquad t>0.\ee This gives the result for $q=2$. Let $0<q< 2$. By
Bernstein's theorem \cite[Chapter 18, Sec. 4]{F}, there is a
non-negative finite measure $\mu_q$ on $[0,\infty)$ so that $
e^{-z^{q/2}}=\int_0^\infty e^{-tz}\,d\mu_q (t)$, $z\in [0,\infty)$.
Replace $z$ by $|y|^2$ to get \be\label{751}
e^{-|y|^{q}}=\int_0^\infty e^{-t|y|^2}\,d\mu_q (t).\ee Then
(\ref{75}) yields \bea
 \gam_{q,\ell}(\eta)&=&\int_{\rl}e^{i y\cdot
\eta}dy\int_0^\infty e^{-t|y|^2}\,d\mu_q (t)= \int_0^\infty d\mu_q
(t)\int_{\rl}e^{i y\cdot \eta} e^{-t|y|^2}\,dy\nonumber\\&=&
\pi^{\ell/2}\int_0^\infty t^{-\ell/2}e^{-|\eta|^2/4t}\,d\mu_q
(t)>0.\nonumber\eea The Fubini theorem is applicable here, because,
by (\ref{751}),
$$
\int_{\rl}|e^{i y\cdot \eta}|dy\int_0^\infty
e^{-t|y|^2}\,d\mu_q=\int_{\rl} e^{-|y|^{q}}dy<\infty.$$
\end{proof}

Our next concern is the behavior of $\gam_{q,\ell}(\eta)$ when
$|\eta| \to \infty $.  If $q$ is
 even, then $e^{-|\cdot|^q}$ is a Schwartz
 function and therefore, $\gam_{q,\ell}$ is infinitely smooth
 and rapidly decreasing. In the general case, we have the following.
\begin{lemma}\label{gql} For any  $q>0$, \be\label{73p}
\lim\limits_{|\eta| \to \infty}|\eta|^{\ell
+q}\gam_{q,\ell}(\eta)=2^{\ell +q}\pi^{\ell/2-1}\Gam (1+q/2)\Gam
((\ell+q)/2)\, \sin (\pi q/2).\ee
\end{lemma}
\begin{proof} For
$\ell=1$, this statement can be found in \cite [Chapter 3, Problem
154]{PS} and  in \cite [p. 45]{K3}. In the general case, the proof
is more sophisticated  and relies on the properties of Bessel
functions. By the well-known formula for the Fourier transform of a
radial function (see, e.g., \cite{SW}), we write
$\gam_{q,\ell}(\eta)=I(|\eta|)$, where \bea
I(s)&=&(2\pi)^{\ell/2}s^{1-\ell/2}\int_0^\infty e^{-r^{q}}
r^{\ell/2} J_{{\ell/2-1}} (rs)\,
dr\nonumber\\&=&(2\pi)^{\ell/2}s^{-\ell}\int_0^\infty e^{-r^{q}}
\frac{d}{dr}\,[(rs)^{\ell/2} J_{{\ell/2}} (rs)]\, dr.\nonumber\eea
Integration by parts yields $$
I(s)=q(2\pi)^{\ell/2}s^{-\ell/2}\int_0^\infty e^{-r^{q}}
r^{\ell/2+q-1} J_{{\ell/2}} (rs)\, dr.$$ Changing variable $z=s^q
r^q$, we obtain
$$ s^{\ell +q}I(s)=(2\pi)^{\ell/2} A(s^{1/q}), \qquad A(\del)=
\int_0^\infty e^{-z\del} z^{\ell/2q} J_{\ell/2} (z^{1/q})\,dz.
$$
We actually have to compute the limit $A_0=\lim\limits_{\del \to 0}
A(\del)$. To this end, we invoke Hankel functions $H_\nu^{(1)} (z)$,
so that $ J_\nu (z)=Re \,H_\nu^{(1)} (z)$ if $z$ is real \cite{Er}.
Let $h_\nu (z)=z^\nu H_\nu^{(1)} (z)$. This is a single-valued
analytic function in the $z$-plane with cut $(-\infty, 0]$. Using
the  properties of the Bessel functions \cite{Er}, we  get
 \be\label {as}
\lim\limits_{z \to 0}h_\nu (z)=2^\nu \Gam (\nu)/\pi i,\ee \be\label
{as1} h_\nu (z) \sim \sqrt{2/\pi} \, z^{\nu -1/2}e^{iz-\frac{\pi i
}{2}(\nu +\frac{1}{2})}, \qquad z \to \infty.\ee Then we write
$A(\del)$ as $ A(\del)= Re \,\int_0^\infty e^{-z\del} h_{\ell/2}
(z^{1/q})\,dz$ and change the line of integration from $[0,\infty)$
to $\ell_\theta=\{z: z=re^{i\theta}, \; r>0\}$ for small $\theta<\pi
q/2$. By Cauchy's theorem, owing to (\ref{as}) and (\ref{as1}), we
obtain $ A(\del)= Re \,\int_{\ell_\theta} e^{-z\del} h_{\ell/2}
(z^{1/q})\,dz$. Since for $z=re^{i\theta}$, $ h_{\ell/2}
(z^{1/q})=O(1)$ when $r=|z|\to 0$ and $ h_{\ell/2}
(z^{1/q})=O(r^{(\ell -1)/2q} e^{-r^{1/q}\sin (\theta /q)})$ as $r\to
\infty$, by the Lebesgue theorem on dominated convergence, we get $
A_0=Re \,\int_{\ell_\theta}  h_{\ell/2} (z^{1/q})\,dz$. To evaluate
 the last integral, we again use analyticity and replace
$\ell_\theta$ by $\ell_{\pi q/2}=\{z: z=re^{i\pi q/2}, \; r>0\}$ to
get $$ A_0=Re \,\Big [e^{i\pi q/2}\int_0^\infty   h_{\ell/2}
(r^{1/q}e^{i\pi/2})\,dr\Big ].$$ To finalize calculations, we invoke
McDonald's function $K_\nu (z)$ so that
$$
h_\nu (z)=z^\nu H_\nu^{(1)} (z)=-\frac{2i}{\pi}(z e^{-i\pi/2})^\nu
K_\nu (z e^{-i\pi/2}).$$ This gives
$$
A_0=\frac{2q}{\pi}\, \sin (\pi q/2) \int_0^\infty s^{\ell/2 +q-1}
K_{\ell/2} (s)\, ds.
$$ The last integral can be explicitly evaluated by the formula 2.16.2
(2) from \cite {PBM}, and we obtain the result.
\end{proof}

Now we can proceed to studying $(q,\ell)$-balls $ B^n_{q,\ell}$; see
(\ref {ball}).
 There is an intimate connection
between  geometric properties of the balls $B^n_{q,\ell}$ and the
Fourier transform of the power function $||\cdot||_{q,\ell}^p$.  The
case $q=2$ is  well-known and associated with  Riesz potentials;
see, e.g., \cite{St}. The relevant case of $\ell^n_q$-balls, which
agrees with $\ell=1$ was considered in Example \ref{ngiv}.
\begin{lemma}\label{pdd} Let $q>0, \; \xi =(\xi', \xi'')\in \rn$,
 $ \gam_{q,\ell}(\xi'')$ and $ \gam_{q,n-\ell}(\xi')$
be the functions of the form (\ref{gaql}). We define

 \be\label{ftr} h_{p,q,\ell} (\xi)=\frac{q}{\Gam
(-p/q)}\int_0^\infty
t^{n+p-1}\,\gam_{q,n-\ell}(\xi't)\,\gam_{q,\ell}(\xi''t)\, dt.\ee

\noindent {\rm (i)} Let $\xi'\neq 0$ and  $\xi''\neq 0$. If $q$ is
 even, then the integral (\ref{ftr}) is absolutely convergent for
 all $p>-n$. Otherwise, it is absolutely convergent when
$-n<p<2q$. In these cases, $h_{p,q,\ell}\, (\xi)$ is a locally
integrable function away from the coordinate subspaces $\rl$ and
$\rnl$.

\vskip 0.3truecm

\noindent {\rm (ii)} If $-n<p<0$, then $h_{p,q,\ell}\, (\xi)\in
L^1_{loc}(\rn) \cap \S'(\rn)$ and $(||\cdot||_{q,\ell}^p)^\wedge
(\xi)=h_{p,q,\ell} (\xi)$ in the sense of $\S'$-distributions.
Specifically, for $\vp \in \S(\rn)$, \be \lng h_{p,q,\ell}\,, \hat
\vp\rng= (2\pi)^{n}\lng ||\cdot||_{q,\ell}^p\,, \vp\rng.\ee
\end{lemma}
\begin{proof} (i) For any $0<\e< a<\infty$,
\bea &&\int_{\e<|\xi'|<a} d\xi'\intl_{\e<|\xi''|<a}|h_{p,q,\ell}\,
(\xi;, \xi'')|\, d\xi''\nonumber\\&&\le\frac{q}{|\Gam
(-p/q)|}\intl_0^\infty t^{n+p-1}\,
dt\intl_{\e<|\xi'|<a}|\gam_{q,n-\ell}\,(\xi't)|\,d\xi'\intl_{\e<|\xi''|<a}|\gam_{q,\ell}\,(\xi''t)|\,
d\xi''\nonumber\\&&=\frac{q}{|\Gam (-p/q)|}\intl_0^\infty t^{p-1}\,
dt\intl_{t\e<|z'|<ta}|\gam_{q,n-\ell}\,(z')|\,dz'\intl_{t\e<|z''|<ta}|\gam_{q,\ell}\,(z'')|\,
dz''\nonumber\\
&&=\frac{q}{|\Gam (-p/q)|}\Big (\intl_0^1 +\intl_1^\infty \Big
)(...)=\frac{q}{|\Gam (-p/q)|}(I_1+I_2).\nonumber\eea The first
integral is dominated by
$$
c\,a^n \intl_0^1 t^{n+p-1}\, dt, \quad c=\sig_{n-\ell -1}\sig_{\ell
-1}\max_{z'}|\gam_{q,n-\ell}\,(z')|\,
\max_{z''}|\gam_{q,\ell}\,(z'')|$$ and is finite for $p>-n$. The
second integral can be estimated by making use of Lemma \ref {gql}.
Specifically,  if $q$ is not an even integer, then
$$
I_2\le c_\e\intl_1^\infty t^{p-1}\, dt\intl_{|z'|>t\e}
\frac{dz'}{|z'|^{n-\ell+q}}\intl_{|z''|>t\e}\frac{dz''}{|z''|^{\ell+q}}\le
c_{\e}\intl_1^\infty t^{p-2q-1}\, dt.
$$
If $q$ is even, then $ \gam_{q,\ell}$ and $ \gam_{q,n-\ell}$ are
rapidly decreasing and $ I_2\le c_{\e, a}\int_1^\infty t^{p-2m-1}\,
dt$ for any $m>0$. This gives what we need.

(ii) If $-n<p<0$, the same argument is applicable with $\e=0$. In
this case, $I_2$ does not exceed $ ||\gam_{q,n-\ell}||_1
||\gam_{q,\ell}||_1\int_1^\infty t^{p-1}\, dt$. The latter is finite
when $p<0$, because, by Lemma \ref {gql}, $ \gam_{q,n-\ell}$ and
$\gam_{q,\ell}$ are integrable functions on
 respective spaces. When $\xi \to \infty$, one can readily check that
$h_{p,q,\ell} \,(\xi)=O(|\xi|^m)$ for some  $ m>0$, and therefore,
$h_{p,q,\ell}\in \S'(\rn)$.

 To compute the Fourier transform
 $(||\cdot||_{q,\ell}^p)^\wedge (\xi)$,
we replace  $||x||_{q,\ell}^p$ by the formula
$$
||x||_{q,\ell}^p=\frac{q}{\Gam (-p/q)}\intl_0^\infty
t^{p-1}\,e^{-|x'/t|^q-|x''/t|^q}\, dt, \qquad p<0,$$ and note that
the Fourier transform of the function $x \to
e^{-|x'/t|^q-|x''/t|^q}$ is just $
\gam_{q,n-\ell}\,(\xi't)\,\gam_{q,\ell}\,(\xi''t)$. Then \bea \lng
||\cdot||_{q,\ell}^p)^\wedge\,, \hat \vp\rng&=& (2\pi)^{n}\lng
||\cdot||_{q,\ell}^p\,, \vp\rng\nonumber\\&=&
\frac{(2\pi)^{n}q}{\Gam (-p/q)}\intl_0^\infty t^{p-1}\, dt
\intl_{\rn}e^{-|x'/t|^q-|x''/t|^q} \overline{\vp (x)}\,
dx\nonumber\\&=&\frac{q}{\Gam (-p/q)}\intl_0^\infty t^{n+p-1}\, dt
\intl_{\rn}\gam_{q,n-\ell}\,(\xi't)\,\gam_{q,\ell}\,(\xi''t)\,\overline{\hat
\vp (\xi)}\, d\xi.\nonumber\eea Interchange of  the order of
integration in this argument can be easily justified using absolute
 convergence of integrals under consideration.
\end{proof}
\begin{theorem} If $\,  0<q\le 2$, $0<\ell<n$, then $B^n_{q,\ell}$ is
 a $\lam$-intersection body for any $0<\lam <n$.
\end{theorem}
\begin{proof}
Owing to  Lemma \ref{pdd0}, the function (\ref{ftr}) (with $p$
replaced by $-\lam$) is positive, and therefore, by Lemma \ref{pdd},
 $||\cdot||_{q,\ell}^{-\lam}$ represents a positive definite
 distribution. Now the result follows by Theorem \ref{aprr0}.
\end{proof}

Consider the case $q>2$. In this case $B^n_{q,\ell}$ is convex, and,
owing to Example \ref{ninn}, $B^n_{q,\ell}\in \I^n_\lam$ for all
$n-3\le \lam <n$. What about $\lam <n-3$?  This case is especially
intriguing.
\begin{proposition}\label{nakz}
 If $q>2$ and $0<\lam <\max (n-\ell, \ell)-2$, then
$||\cdot||_{q,\ell}^{-\lam}$ is not a positive definite distribution
and therefore,
 $B^n_{q,\ell}\not\in \I^n_\lam$ .
\end{proposition}
\begin{proof}  Let $0<\lam<n-\ell -2$ and suppose the contrary, that
$B^n_{q,\ell}\in \I^n_\lam$. Consider the section of $B^n_{q,\ell}$
by the $(n-\ell +1)$-dimensional  plane $\eta=\bbr e_n\oplus\rnl$.
By Theorem \ref{tyif},  $B^n_{q,\ell} \cap \eta \in \I^{n-\ell
+1}_{\lam}$ in $\eta$, and therefore
$$
||x_n e_n + x''||^{\lam}_{q,\ell}=(|x_n|^q +|x''|^q)^{-\lam/q}
$$
is a positive definite distribution in  $\eta$. By the second
derivative text (see \cite[Theorem 4.19]{K3}) this is impossible if
$0<\lam<n-\ell -2$. A similar contradiction can be obtained if we
 assume $0<\lam<\ell -2$ and consider the section of $B^n_{q,\ell}$ by the $(\ell
+1)$-dimensional  plane $\bbr e_1\oplus\rl$.
\end{proof}
Proposition \ref{nakz} can be proved  without using  the second
derivative text and Theorem \ref{tyif} on sections of
$\lam$-intersection bodies; see \cite{R4}. The bounds for $\lam$
appear to be the same.

\noindent {\bf Open problem.} {\it  Let  $q>2, \; \ell>1$. Is
$B^n_{q,\ell}$ a
 $\lam$-intersection body if} $\max (n-\ell, \ell)-2 <\lam <n-3$?

 This problem does not occur in the case $\ell=1$ as in Example
 \ref{ngiv}.

\section {The generalized cosine transforms and comparison of
volumes} For $1<i<n$, let $\vol_i ( \cdot)$ denote the
$i$-dimensional volume function. Suppose that $i$ is fixed, and let
$A$ and $B$ be  o.s.
 convex bodies  in $\bbr^n$ satisfying \be\label{cons} \vol_i(A \cap \xi)
\le \vol_i(B \cap \xi) \quad \forall \xi \in G_{n,i}. \ee Does it
follow that \be\label{cons1} \vol_n(A) \le \vol_n(B) \quad \text{\rm
?} \ee

This question is known as the  Generalized Busemann-Petty Problem
(GBP); see  \cite {G}, \cite {RZ}, \cite {Z1}.
\begin{theorem}\label{bhb} If GBP
(\ref{cons})-(\ref{cons1}) has an affirmative answer,  then every
smooth origin-symmetric  convex body with positive curvature  in
$\bbr^n$ is an $(n-i)$-intersection body.
\end{theorem}
\begin{proof}
Suppose that $B$ is an o.s. convex body in $\bbr^n$ so that the
radial function $\rho_B$ is infinitely smooth, the boundary of $B$
has  a positive curvature and $B\notin \I^n_{n-i}$. By Definition
\ref{dffl},  there is a function  $\vp\in \D_{e}(S^{n-1})$, which is
negative on some open origin-symmetric set $\Om
 \subset S^{n-1}$ and such that $ \rho_B^{n-i}=M^{1+i-n} \vp$. We
 choose a function $h\in
\D_{e}(S^{n-1})$ so that $h \not\equiv 0$, $h(\theta)\ge 0$ if
$\theta \in \Om$ and $h(\theta)\equiv 0$ otherwise. Define an o.s.
smooth  body $A$ by $ \rho_A^{i}=\rho_B^{i}-\e M^{1-i} h$, $\e>0$.
If $\e$ is small enough, then $A$ is convex. Since by (\ref{con}),
$R_i M^{1-i} h=c\,
 R_{n-i,\perp}^{0}h \ge 0$, then $R_i\rho_A^{i}\le R_i\rho_B^{i}$,
 which gives (\ref{cons}). On the other hand, by (\ref{st}),
$$ ( \rho_B^{n-i},\rho_B^{i} -\rho_A^{i})=\e(M^{1+i-n} \vp, M^{1-i}
h)=\e(\vp, h)<0,$$ or $( \rho_B^{n-i},\rho_B^{i})<(
\rho_B^{n-i},\rho_A^{i})$. By H\"older's inequality, this implies
 $ \vol_n(B)<\vol_n(A)$, which contradicts (\ref{cons1}).
\end{proof}

\begin{remark} As we noted in Introduction, Theorem 8.1 is not new,
and its proof given in \cite{K4} relies on a sequence of deep facts
from functional analysis. The proof presented above is much more
elementary and constructive. For instance, it allows us to keep
invariance properties of the bodies under control. This advantage
was essentially used in our paper \cite{R4}.
\end{remark}

Theorem \ref{bhb} and  Proposition \ref{nakz} imply the following
\begin{corollary} Let $1\le \ell \le n/2; \; i>\ell +2$,
$B=B^n_{4,\ell}$ {\rm (see (\ref{ball}))}. Then there is a smooth
o.s. convex body $A$ in $\bbr^n$ so that (\ref{cons}) holds but
(\ref{cons1}) fails.
\end{corollary}
Setting $\ell =1$ in this statement, we obtain the well-known
Bourgain-Zhang theorem, which states that GBP has a negative answer
when $3<i<n$; see \cite {BZ}, \cite {K3}, \cite {RZ} on this
subject. For $i=2$ and $i=3$ ($n\ge 5$) the GBP is still open. An
 affirmative answer in these cases was obtained in \cite{R4} for
 bodies having a certain additional symmetry.

\section {Appendix}

 Every positive distribution $F\in \S'(\bbr^n)$
 is given by a tempered non-negative
 measure $\mu$, i.e.,  $\lng F, \phi\rng=\int \phi (x) d\mu (x)$;
 see, e.g., \cite[p.147]{GV}).
 For convenience of the reader, we present a similar fact for the
 sphere.
\begin{theorem}\label{sm}
A distribution $f \in \D'(S^{n-1})$ is positive if and only if there
is a measure $\mu \in \M_+ (S^{n-1})$ such that
$$(f,\vp)=\int_{S^{n-1}}\vp (\theta)
d\mu (\theta)\qquad \forall \vp \in  \D(S^{n-1}).$$
\end{theorem}
\begin{proof} This statement is  known, however, we  could not find precise
reference and decided to give a  proof for convenience of the
reader. The ``if" part is obvious. To prove the `` only if" part, we
write a test function $\vp\in  \D(S^{n-1})$ as a sum  $\vp= \vp_1 +
i \vp_2$, where $\vp_1=Re \, \vp, \; \vp_2=Im \,  \vp$. Since
$-||\vp||_{C(S^{n-1})} \le \vp_j\le ||\vp||_{C(S^{n-1})}$, $j=1,2$,
and $f$ is positive, then
$$-(f,1)\,||\vp||_{C(S^{n-1})} \le (f,\vp_j)\le (f,1)\,||\vp||_{C(S^{n-1})},$$
and therefore, $|(f,\vp)|\le |(f,\vp_1)|+|(f,\vp_2)|\le
2(f,1)\,||\vp||_{C(S^{n-1})}$. Since $\D(S^{n-1})$ is dense in
$C(S^{n-1})$, then  $f$ extends as a linear continuous functional
$\tilde f$ on $C(S^{n-1})$ and, by the Riesz theorem, there is a
measure $\mu$ on $S^{n-1}$ such that $(\tilde
f,\om)=\int_{S^{n-1}}\om (\theta) d\mu (\theta)$ for every $\om \in
C(S^{n-1})$. In particular, $(f,\vp) =(\tilde
f,\vp)=\int_{S^{n-1}}\vp (\theta) d\mu (\theta)$ for every $\vp \in
\D(S^{n-1})$. By taking into account that every non-negative
function  $\om \in C(S^{n-1})$ can be uniformly  approximated by
non-negative functions $\vp_k \in \D(S^{n-1})$ (for instance, by
Poisson integrals of $\om$), we get
$$
\int_{S^{n-1}}\om (\theta)d\mu (\theta)=\lim\limits_{k\to
\infty}\int_{S^{n-1}}\vp_k(\theta)d\mu (\theta)=\lim\limits_{k\to
\infty}(f,\vp_k) \ge 0.$$ The latter means that $\mu$ is non-negative.
\end{proof}


\begin{thebibliography}{[ASMR]}

\bibitem [BZ] {BZ} J. Bourgain,  G. Zhang,  On a generalization of the
Busemann-Petty problem, Convex geometric analysis (Berkeley, CA,
1996), 65--76, Math. Sci. Res. Inst. Publ., 34, Cambridge Univ.
Press, Cambridge, 1999.

%\bibitem   [BIN] {BIN} O.V. Besov, V. P. Illin, and C. M. NikollskiY, Integral
%Representations and Embedding Theorems. [in Russian], Nauka, Moscow,
%1975.

\bibitem [Er] {Er}  A. Erd\'elyi (Editor),  Higher transcendental functions, Vol.
  II, McGraw-Hill, New York, 1953.

%\bibitem [1] {Es} G. I. Eskin,  Boundary value problems for elliptic
%pseudodifferential equations, Amer. Math. Soc., Providence, R.I.,
%1981.

\bibitem [FGW] {FGW} H. Fallert, P. Goodey,  W, Weil,
\textit{Spherical projections and centrally symmetric Sets},
Advances in Math., \textbf{129} (1997), 301--322.

\bibitem [F] {F} W. Feller, An introduction to probability theory
and its application, Wiley \& Sons, New York, 1971.

\bibitem [G] {G}  R.J. Gardner,  Geometric tomography, Cambridge University
Press, New York, 1995; updates in http://www.ac.wwu.edu/~gardner/.

\bibitem [GGG] {GGG}
I. M. Gel'fand, S. G. Gindikin, and M. I. Graev,  Selected topics in
integral geometry, Translations of Mathematical Monographs, AMS,
Providence, Rhode Island, 2003.

\bibitem [GS] {GS}  I. M. Gelfand, G.E. Shilov,
Generalized functions, vol. 1, Properties and  Operations, Academic
Press, New York,
 1964.

\bibitem [GV] {GV}  I. M. Gelfand, N. Ya. Vilenkin,
Generalized functions, vol. 4, Applications of harmonic analysis,
Academic Press, New York,
 1964.

\bibitem [GLW] {GLW}  P. Goodey, E. Lutwak, W. Weil, Functional analytic characterizations
of classes of convex bodies,  Math. Z. \textbf{222}  (1996),
 363--381.


\bibitem [GW] {GW}  P. Goodey,  W. Weil,   Intersection bodies and ellipsoids.
Mathematika,  \textbf{42}  (1995),  295--304.



%\bibitem  [Gri] {Gri}  E. L. Grinberg, Spherical harmonics and integral geometry on
%projective spaces, Trans. Amer. Math. Soc. \textbf{279} (1983),
% 187–203.

\bibitem  [GZ]{GZ}  E.L. Grinberg,  G.  Zhang, Convolutions, transforms, and
convex bodies,  Proc. London Math. Soc. (3), \textbf{78}
  (1999),  77--115.

\bibitem [He] {He} S. Helgason,   The Radon transform,
Birkh\"auser, Boston, Second edition, 1999.

\bibitem [K1] {K5}  A. Koldobsky, Intersection bodies in $\bbr^4$,
Adv. Math., \textbf {136} (1998), 1-14.

\bibitem [K2] {K2} \bysame,  A generalization of the Busemann-Petty
problem on sections of convex bodies,  Israel J. Math. \textbf {110}
(1999), 75--91.

\bibitem [K3] {K4} \bysame, A functional analytic approach to
intersection bodies,  Geom. Funct. Anal.,  \textbf {10} (2000),
1507--1526.

\bibitem [K4] {K3} \bysame,
Fourier analysis in convex geometry,
 Mathematical Surveys and Monographs, \textbf {116}, AMS, 2005.

\bibitem [Le] {Le} C. Lemoine, Fourier transforms of homogeneous distributions,
 Ann. Scuola Norm. Super. Pisa Sci. Fis. e Mat., \textbf {26}
(1972), No. 1, 117--149.

\bibitem [Lu] {Lu} E. Lutwak, Intersection bodies and dual mixed volumes,
 Adv. in Math. \textbf {71} (1988),  232--261.

\bibitem [Mi1] {Mi1} E. Milman, Generalized intersection bodies,
J. Funct. Anal., \textbf{240} (2006), 530--567.

\bibitem [Mi2] {Mi2} \bysame, Generalized intersection bodies are not
equivalent, math.FA/0701779.

\bibitem [M\"u] {M}  Cl. M\"uller, Spherical harmonics,
 Springer, Berlin, 1966.

\bibitem [Ne] {Ne} U. Neri,  Singular integrals,
 Springer, Berlin, 1971.

\bibitem [PS] {PS} G. Polya, G. Szego, Aufgaben und lehrsatze aus
der analysis, Springer-Verlag, Berlin-New York, 1964.

\bibitem [PBM] {PBM} A. P. Prudnikov,   Y. A. Brychkov,   O. I. Marichev,
  Integrals and series:  special
functions,  Gordon and Breach Sci. Publ., New York - London, 1986.

\bibitem [R1] {R3}  B. Rubin,   Inversion of fractional integrals related
to the spherical Radon transform,  Journal of Functional Analysis,
\textbf{157} (1998), 470--487.

\bibitem [R2] {R1} \bysame, Inversion formulas for the spherical
 Radon transform and the generalized cosine transform,  Advances in Appl.
 Math. \textbf{ 29} (2002), 471--497.

\bibitem [R3] {R2} \bysame, Notes on Radon transforms in integral
geometry, Fractional Calculus and Applied Analysis, \textbf{6}
(2003), 25--72.


\bibitem [R4] {R4} \bysame, The lower dimensional Busemann-Petty problem for
bodies with the generalized axial symmetry, math.FA/0701317.

\bibitem [R5] {R5} \bysame, Generalized cosine transforms and classes of star
bodies, math.FA/0602540.

\bibitem [RZ] {RZ}  B. Rubin, G. Zhang, Generalizations of the
Busemann-Petty problem for sections of convex bodies, J. Funct.
Anal., \textbf{213} (2004), 473--501.

\bibitem [Sa1] {Sa0} S. G. Samko, The Fourier transform of the
functions $Y_m(x/|x|)/|x|^{n+\a}$, Soviet Math. (IZ. VUZ)
\textbf{22} (1978), no. 7, 6--64.


\bibitem [Sa2] {Sa1} \bysame,  Generalized Riesz
potentials and hypersingular integrals with homogeneous
characteristics, their symbols and inversion,
  Proceeding of the Steklov Inst. of Math.,  \textbf{ 2} (1983) , 173--243.

\bibitem [Sa3] {Sa2}  \bysame,  Singular integrals over a sphere and the construction
of the characteristic from the symbol,  Soviet Math. (Iz. VUZ),
\textbf{27} (1983), No. 4, 35--52.

\bibitem [Schn] {Schn}   R. Schneider,  Convex bodies: The Brunn-Minkowski theory,
Cambridge Univ. Press, 1993.


\bibitem [Schw] {Schw} L. Schwartz, Th\'eorie des distributions,
Tome 1, Paris, Hermann, 1950.

\bibitem [Se] {Se} V.I. Semyanistyi,  Some integral transformations and integral
geometry in an elliptic space,   Trudy Sem. Vektor. Tenzor. Anal.,
\textbf{12}  (1963), 397--441 (Russian).

\bibitem [St] {St}  E. M. Stein, Singular integrals and differentiability
properties of functions, Princeton Univ. Press, Princeton, NJ, 1970.

\bibitem [SW] {SW} E.M. Stein,   G. Weiss, Introduction to Fourier analysis on
Euclidean spaces,  Princeton Univ. Press, Princeton, NJ,   1971.

\bibitem [Str1] {Str1} R.S. Strichartz,  Convolutions with kernels having
singularities on a sphere, Trans. Amer. Math. Soc., {\bf 148}
(1970), 461--471.

\bibitem [Str2] {Str2} \bysame,  $L^p$-estimates for Radon transforms in
Euclidean  and non-euclidean spaces,  Duke Math. J., \textbf{ 48}
(1981), 699--727.


%\bibitem [Vl] {Vl} V.S. Vladimirov, Generalized functions in
%mathematical physics, Moskva, Nauka, 1979 (Russian).

\bibitem [Z1] {Z1}  G. Zhang, Sections of convex bodies,
Amer. J. Math., \textbf{118} (1996),  319--340.

\bibitem [Z2] {Z2} \bysame,   A positive solution to
the Busemann-Petty problem in $\bbr^4$,  Ann. of Math. (2), 149
 (1999),  535--543.

\end{thebibliography}
\end{document}